\documentclass[11pt,oneside]{amsart}
\usepackage[english]{babel}
\usepackage{fullpage}
\usepackage{amssymb,pstricks,amscd,epsfig}
\usepackage{graphicx}

\usepackage{psfrag}
\usepackage{setspace}
\usepackage{paralist}

\usepackage{amsmath, amsthm, amssymb}
\usepackage{pdfsync}
\usepackage[latin1]{inputenc}
\usepackage{stmaryrd}
\usepackage{color}
\usepackage{setspace}
\usepackage[matrix,arrow,tips,curve]{xy}

\numberwithin{equation}{section}

\makeatletter
 \renewcommand\section{\@startsection {section}{1}{\z@}%
     {-4.5ex \@plus -1ex \@minus -.2ex}%
     {2.3ex \@plus.8ex}%
    {\centering\scshape}}

\newcommand\E{\mathcal{E}}
\newcommand\F{\mathcal{F}}
\newcommand\G{\mathcal{G}}
\newcommand\I{\mathcal{I}}

\newcommand\V{\mathcal{V}}
\newcommand\U{\mathcal{U}}

\newcommand\W{\mathcal{W}}

\newcommand\X{\mathcal{X}}

\newcommand\PP{\mathbb{P}}
\newcommand\CC{\mathbb{C}}
\newcommand\ZZ{\mathbb{Z}}

\newcommand\QQ{\mathbb{Q}}

\newcommand\VV{\mathbf{V}}

\newcommand\Def{\mathrm{Def}}
\newcommand\Pic{\operatorname{Pic}}
\newcommand\Hom{\operatorname{Hom}}
\DeclareMathOperator{\Aut}{Aut}
\DeclareMathOperator{\autr}{\mathbf{Aut}_{r}}
\DeclareMathOperator{\autone}{\mathbf{Aut}_{1}}
\DeclareMathOperator{\autb}{\mathbf{Aut}}
\DeclareMathOperator{\isob}{\mathbf{Iso}}
\DeclareMathOperator{\Iso}{Iso}

\DeclareMathOperator{\End}{End}

\DeclareMathOperator{\Stab}{Stab}
\DeclareMathOperator{\Amp}{Amp}
\newcommand\Ext{\operatorname{Ext}}
\newcommand\rep{\operatorname{Rep}}
\newcommand\spec{\operatorname{Spec}}
\newcommand\supp{\operatorname{Supp}}

\renewcommand\End{\operatorname{End}}

\newcommand\dext{\operatorname{ext}}

\DeclareMathOperator{\GL}{GL}
\DeclareMathOperator{\gl}{\mathfrak {gl}}
\DeclareMathOperator{\Defo}{Def}
\DeclareMathOperator{\Quot}{Quot}

\DeclareMathOperator{\rk}{rk}
\newcommand\wt{\widetilde}
\newcommand\wh{\widehat}

\newcommand\Shom{{\mc H}om}

\newtheorem{theorem}{Theorem}[section]
\newtheorem{prop}[theorem]{Proposition}
\newtheorem{lem}[theorem]{Lemma}
\newtheorem{defin}[theorem]{Definition}
\newtheorem{rem}[theorem]{Remark}
\newtheorem{notation}[theorem]{Notation}
\newtheorem{cor}[theorem]{Corollary}
\newtheorem{example}[theorem]{Example}

\newtheorem{theorem-defin}[theorem]{Theorem--Definition}

\newcommand{\be}{\begin{equation}}
\newcommand{\ee}{\end{equation}}

\newcommand{\mc}{\mathcal}
\renewcommand{\L}{\mc{L}}

\newcommand{\ov}{\overline}

\newcommand{ \ff} { \frac }

\newcommand{\mk}{\mathfrak}

\DeclareMathOperator{\tr}{tr}

\DeclareMathOperator{\ch}{ch}
\DeclareMathOperator{\td}{td}

\DeclareMathOperator{\Spec}{Spec}
\DeclareMathOperator{\gr}{gr}
\newcommand\Rep{\operatorname{Rep}}

\DeclareMathOperator{\slope}{slope}

\begin{document}

\title[Singularities and Quivers]{
Singularities of moduli spaces of sheaves on K3 surfaces and Nakajima quiver varieties }

\author{E.~Arbarello and  G.~Sacc\`a}

\address{Dipartimento di Matematica Guido Castelnuovo, Universit\`a di Roma Sapienza, 
Piazzale A. Moro 2, 00185 Roma, Italia. }
\email{ea@mat.uniroma1.it}

\address{ Department of Mathematics,
Stony Brook University,
Stony Brook, NY
11794-3651, USA}
\email{giulia.sacca@stonybrook.edu}

\maketitle

\begin{abstract}

The aim of this paper is to study the singularities of certain moduli
spaces of sheaves on K3 surfaces by means of Nakajima quiver
varieties. The singularities in question arise from the choice of a
non--generic polarization, with respect to which we consider stability,
and admit natural symplectic resolutions corresponding to choices of
general polarizations. For sheaves that are pure of dimension one, we show that these moduli spaces are, locally around a
singular point, isomorphic to a quiver variety and that, via this isomorphism, the natural symplectic
resolutions correspond to variations of GIT quotients of the quiver
variety. 
\end{abstract}

\tableofcontents

\section{Introduction}\label{intro}

A normal variety $X$ is said to have symplectic singularities \cite{Beauville-symplectic} if its smooth locus $X^{sm}$ carries a holomorphic symplectic form $\sigma$ having the property that,  for any resolution $f: Y \to X$, the  pull-back of $\sigma$ to  $f^{-1}(X^{sm})$ extends to a holomorphic form $\sigma_Y$ on $Y$. When this is the case, $X$ is called a {\it symplectic variety}. A resolution $f: Y \to X$ of a symplectic variety is called \emph{symplectic} if, in addition, the holomorphic $2$-form $\sigma_Y$ is non-degenerate. In particular, a symplectic resolution is crepant. Symplectic resolutions are rare: for example, $\mathbb C^{2n} \slash \pm1$ with the standard symplectic form on the smooth locus is a symplectic singularity, but it admits a symplectic resolution if and only if $n=1$.

Examples of symplectic varieties and symplectic resolutions come from both representation theory and  the theory  of moduli spaces of sheaves on K3 or abelian surfaces. 
Among the symplectic varieties coming from representation theory, we find the nilpotent cone of a complex semisimple Lie algebra and its Springer resolution, the quotients of $\mathbb C^2$ by a finite group of symplectic automorphism and their minimal resolutions, and Nakajima quiver varieties.  
Regarding  moduli spaces of sheaves on a K3 surface, 
their symplectic singularities come from two sources, when the Mukai vector is not primitive, or when the polarization (more generally, the stability condition) is not general. We  explain this in  Section \ref{moduli-pure-dim-one}. In \cite{Nakajima-Hilb}, Nakajima showed that the Hilbert--Chow morphism, from the Hilbert scheme of points on a holomorphic symplectic surface to the symmetric product of the surface itself, can be described in terms of quiver varieties. This fruitful interaction between quiver varieties and Hilbert schemes of points on surfaces has generated several results, especially on the cohomology and Chow groups of Hilbert schemes. One of the aims of the present article is to generalize Nakajima's description to other moduli spaces and this is the first step in that direction.

Two particular cases of singularities due to a non-primitive Mukai vector were studied by O'Grady  \cite{OGrady99}, \cite{OGrady03}.  Through this study, he discovered two new examples of irreducible holomorphic symplectic manifolds by exhibiting symplectic resolutions of two singular moduli spaces on a K3 surface and on an abelian surface, respectively. Inspecting O'Grady's construction, Kaledin, Lehn, and Sorger showed, in their inspiring paper \cite{Kaledin-Lehn-Sorger06}, that in the remaining cases with non-primitive Mukai vector the corresponding moduli space has  no symplectic resolution. Our aim is to continue their investigation, and to study the case when the singularities of a moduli space of sheaves arise from the choice of a non-generic polarization. In certain cases, moving slightly the polarization to a general one induces a symplectic resolution of the singular moduli space. Our specific purpose is to find a local analytic model of these singularities, as well as of their modular symplectic resolutions.

The case we will be studying is the one of pure dimension one sheaves on a K3 surface $S$. By definition, these are sheaves whose support, as well as that of any non-trivial sub-sheaf, has dimension one. Let us briefly explain the reasons for this choice.
Let $v\in H^*_{alg}(S,\ZZ)$ be the Mukai vector of a pure dimension one sheaf on $S$.  Yoshioka showed that the ample cone $\Amp(S)$ admits a finite wall and chamber structure relative to $v$. If $v$ is primitive, then for polarizations lying in a chamber
(i.e., not on a wall) the moduli space $M_H(v)$ is smooth. On the contrary, if a polarization $H_0$ is contained in a wall, then the corresponding moduli space $M_{H_0}(v)$ is singular. We choose to study the case of pure dimension sheaves because if $H$ lies in a chamber containg $H_0$ in its closure, then there is natural \emph{regular} morphism $h: M_H(v)  \to M_{H_0}(v)$, which is a symplectic resolution. In higher rank, this is not always the case, and one needs to look instead at resolutions arising from Matsuki--Wentorth twisted stability or from Bridgeland stability conditions. For example, in the case of ideal sheaves our methods recover Nakajima's quiver description of the Hilbert--Chow morphism. These are the next steps in our program and will be addressed in a separate work.

 To state our main theorem we need to introduce some notation. 
A quiver, denoted by $Q$, is an oriented graph. Let $I=\{1, 2, \dots, s\}$ be the set of vertices of $Q$ and denote by $E$ the set of edges. For an edge $e \in E$, we denote by  $s(e)$ and  $t(e) \in I$  the source and target of $e$, respectively. Given a 
 dimension vector $\mathbf n=(n_1, \dots, n_s) \in \mathbb Z^s_{\ge 0}$, we choose, for each $i=1,\dots s$ a complex $n_i$--dimensional vector space $V_i$ and we  let
\[
\Rep(\ov Q, \mathbf n) =\bigoplus_{e \in E} \Hom(V_{s(e)}, V_{t(e)}) \oplus \Hom(V_{t(e)}, V_{t(e)}) 
\]
be the space of $\mathbf n$--dimensional representions of the double quiver $\ov Q$ (defined in Section \ref{quivers}). The group $G:=G( {\mathbf n})=\prod GL(V_i)$
acts on $\Rep(\ov Q,\mathbf n)$ via conjugation and $\Rep(\ov Q,\mathbf n)$ is naturally equipped with a $G$-invariant
symplectic form. This is the context in which one can define a   {\it moment map}, with values in the Lie 
algebra $\mathfrak g$ of $G$
\[
\mu: \Rep(\ov Q,\mathbf n) \to \mathfrak g, \quad \quad \sum (x_e, y_e) \mapsto \sum [x_e, y_e]
\] 
Via the moment map, it is possible to perform symplectic reduction,
the essence of which is  that the quiver variety $\mathfrak M_0:=\mu^{-1}(0) \sslash G$ is a symplectic variety. When $\mathbf n$ is primitive, a symplectic resolution of $\mathfrak M_0$ can often be achieved via GIT. More precisely, 
let $\chi \in \Hom(G, \mathbb C)$ be a rational character of $G$. By considering the GIT quotient $\mathfrak M_\chi:=\mu^{-1}(0)\sslash_\chi G $ we get a projective morphism
\begin{equation} \label{xi}
\xi: \mathfrak M_\chi \to \mathfrak M_0,
\end{equation}
which, in many cases, is birational.
 In \cite{Nakajima-branching}, Nakajima  shows that there is a wall and chamber decomposition of $\Hom(G, \mathbb C)\otimes \mathbb Q$, so that if $\chi$ is chosen in a chamber then (\ref{xi}) is a symplectic resolution.

 We can now state the main theorem (Theorem \ref{main}).
 
\begin{theorem} \label{main-intro}Let $v$ be a primitive Mukai vector of a pure dimension one sheaf on $S$. For any singular point $x \in M_{H_0}(v)$ there exists a quiver $Q$ and a dimension vector $\mathbf n$ such that
\begin{enumerate}
\item[i)] There is a local isomorphism $\psi: (\mathfrak{M}_0,0) \cong (M_{H_0}(v), x)$;
\item[ii)] For every polarization $H$ in a chamber containing $H_0$ in its closure, there is a character $\chi_H$ in a chamber of $\Hom(G, \mathbb C)\otimes \mathbb Q$ such that the symplectic resolutions
\[
\xi: \mathfrak M_{\chi_H}(\mathbf{ n}) \to \mathfrak M_0(\mathbf{ n}), \quad \text{ and } \quad h: M_H(v) \to M_{H_0}(v),
\]
correspond to each other via $\psi$. 
\end{enumerate}
\end{theorem}

\noindent
Let us make a few remarks.

First of all, recall that given a singular point $x \in M_{H_0}(v)$, there is a unique up to isomorphism $H_0$--polystable sheaf $F=\oplus_{i=1}^s F_i^{n_i}$ in the $S$--equivalence class represented by $x$. With this notation, the quiver $Q$ has $s$ vertices, and for every $i<j$, it has $\dim \Ext^1(F_i, F_j)$ edges from $i$ to $j$, and if $i=j$ it has $\dim \Ext^1(F_i, F_i)/2$ loops at the vertex $i$. This can be defined for arbitrary polystable sheaves, but if $F$ is pure of dimension one, then $Q$ is ``essentially'' the dual graph of its support. Also notice that $\Aut(F)=G$.

The heart of the main theorem is item $(ii)$ where the isomorphism in item $(i)$ is lifted to an isomorphism between symplectic resolutions of the two sides, and the wall--and--chamber structure of $\Amp(S)$ is explicitly compared with the one of $\Hom(G, \mathbb C)$.
The assignment $H \mapsto \chi_H$ of part $ii)$ can be chosen to be given by the following formula
\[
G \ni (g_1, \dots, g_s) \mapsto \chi(g_1, \dots, g_s)=\prod_{i=1}^s\det(g_i)^{(D_i\cdot H-D_i \cdot H_0)}, \,\, \text{where}\,\, D_i:=c_1( F_i),
\]
(for a more precise statement see (iii) of Theorem \ref{main}). 

Next, two words about the isomorphism in statement $(i)$ which holds, for any polystable stable sheaf $F$ satisfying the formality property or, more generally, satisfying the quadraticity property we will now discuss.

At any point $x=[F]$, a moduli space $M_{H_0}(v)$ is locally isomorphic to the quotient of the deformation space $\Def_F$ by the automorphism group $G=\Aut(F)$.
The differential graded Lie algebra (dgla) $R\Hom(F, F)$
is said to satisfy the {\it formality property}, if  it is quasi--isomorphic to its cohomology algebra $\Ext^*(F,F)$. When this is the case,
the deformation space $\Def_F$ is isomorphic to a complete intersection of quadrics in $\Ext^1(F, F)$. We call this the {\it quadraticity property }
of  $\Def_F$.
A result that is instrumental in our proof of part $(i)$  is to prove the quadraticty property for  $\Def_F$. 

\begin{theorem} \label{formality-intro} Let $x=[F]\in M_{H_0}(v)$ be a point corresponding to a $H_0$-polystable sheaf $F$  pure of dimension one on $S$. Then the deformation space $\Def_F$ is isomorphic to a complete intersection of quadrics in $\Ext^1(F, F)$. \end{theorem}

We thank Z. Zhang for pointing out to us that in a previous version this theorem was incorrectly stated (referring to formality instead of quadraticity).
Next is a brief description of the contents of the various sections.

In {\bf  Section \ref{moduli-pure-dim-one}} we set up the notation we use for moduli spaces of  pure sheaves of dimension one on a K3 surface, describing how the choice of a Mukai vector induces a wall and chamber structure on the ample cone of $S$.
 
In {\bf  Section \ref{Formality}} we prove Theorem \ref{formality-intro}. Using results by Yoshioka \cite{Yoshioka-FM}
we reduce the proof of  Theorem \ref{formality-intro} to a formality result  for sheaves of positive rank,  due to Zhang \cite{Zhang12}.

{\bf  Section \ref{G-equi-kuran}} is devoted to the study of Kuranishi families for a polystable sheaf on $S$. The formality property for a polystable sheaf $F$ implies that a Kuranishi family is (the completion of) a complete intersection of quadrics
$$
\kappa_2^{-1}(0)\subset\Ext^1(F,F)
$$
as in (\ref{k2 cup}). A subtle point is that the algebraization of this family can be preformed $G$--equivariantly.

In {\bf Section \ref{quivers}} we briefly recall the results on quiver varieties we need for our purposes. This paves the way to understand the GIT partial desingularizations of $\kappa_2^{-1}(0)\sslash G$ in terms of the characters of $G$.

The main theorem (Theorem \ref{main}) is stated in {\bf Section \ref{Statement of the main theorem}}. In this section we also relate very explicitly the wall and chamber structure of the ample cone of $S$, to the wall and chamber structure of $\Hom(G, \mathbb C)\otimes \mathbb Q$. 
The proof of the main theorem is given in  {\bf Section \ref{proof of main theorem}} and  uses the geometry of the Quot scheme, of an \'etale slice around a point corresponding to $F$, and certain natural determinant line bundles.

\subsection*{Acknowledgments} It is a pleasure to heartfully thank the many people we had enjoyable and useful conversations about the content of this paper.
They are: A. Bayer, M. de Cataldo, C. De Concini, D. Fiorenza, D. Huybrechts, R. Laza, M. Lehn, E. Macr\`i, A. Maffei, M. Manetti, E. Markman, C. Procesi, J. Starr, A. Vistoli, K. Yoshioka, Z. Zhang. We are also grateful to H. Nakajima for pointing out the reference \cite{Nakajima-branching}.

We express our gratitude to the Hausdorff Research Institute for Mathematics in Bonn and to the Institute for Advanced Study in Princeton, for generous hospitality during the various phases of preparation of this work. The second named author was member at the Institute for Advance Study during the year 2014--2015 and gratefully acknowlegdes the support of the Giorgio and Elena Petronio Fellowship Fund II and of NSF grant DMS--1128155.

\section{Notation and generalities on moduli spaces of sheaves on a K3  surface}\label{moduli-pure-dim-one}

Throughout this paper $S$ will denote a projective K3 surface. Given a sheaf $F$ on $S$, its Mukai vector $v=v(F)$ is defined by
\[
\ch(F) \sqrt{\td (F)}=(\rk (F), c_1(F), \chi(F)-r) \in H^*_{alg}(S, \ZZ).
\]
The lattice $H^*(S, \ZZ)$ is equipped with the non-degenerate Mukai pairing defined by
\[
v \cdot w= v_1 w_2-v_0 w_2 -v_2 w_0,
\]
for $v=(v_0, v_1, v_2)$ and $w=(w_0, w_1, w_2)$ in $H^*(S, \ZZ)$. If $F$ and $G$ are two coherent sheaves of Mukai vector $v$ and $w$, respectively, then
\[
\chi(F,G)=-v \cdot w.
\] 
In the following, by Mukai vector we will mean an element in $H^*_{alg}(S, \ZZ)$, which is the Mukai vector of some coherent sheaf on $S$.
Given a polarization $H$ in the ample cone $\Amp(S)$, we let $M_H(v)$ be the moduli space of $H$-semistable sheaves with Mukai vector $v$. Here,  semi-stability with respect to a given polarization $H$ means Gieseker semi-stability, defined in terms of the reduced Hilbert polynomial associated to $H$.

We denote by
\[
M^s_H(v) \subset M_H(v)
\]
the locus parametrizing stable sheaves. As proved by Mukai \cite{Mukai}, this is a smooth symplectic variety.
Indeed, given a point $[F] \in M^s_H(v)$, there a canonical identification
\[
T_{[F]}M^s_H(v) = \Ext^1(F,F),
\]
and obstructions to smoothness lie in the trace free part of $\Ext^2(F, F)$ (we will expand on this in Section \ref{G-equi-kuran}, while talking about deformation spaces). By stability and Serre duality,  $\Ext^2(F, F)\cong \CC$, and hence the obstruction space vanishes.
Moreover, when non-empty, we have
\[
\dim M_H(v) =v^2+2.
\]
Finally, the smooth variety $M^s_H(v)$ is endowed with the symplectic form defined at each point by the cup product
\[
\Ext^1(F,F) \times \Ext^1(F,F) \stackrel{\cup}{\longrightarrow} \Ext^2(F, F)=\CC.
\]

Following Yoshioka \cite{Yoshioka}, we make the following definition
\begin{defin} \label{positive v}
We say that a primitive element $v=(v_0, v_1, v_2) \in  H^*_{alg}(S, \ZZ)$  is positive if $v^2 \ge -2$ and one of the following holds:
\begin{itemize}
\item $v_0>0$;
\item $v_0=0$, $v_1$ is effective, and $v_2 \neq 0$;
\item $v_0=v_1=0$ and $v_2>0$.
\end{itemize}
\end{defin}

Following \cite{Bayer-Macri-proj}, Theorem 5.2 we can state
\begin{theorem}[Yoshioka] \label{Yoshioka non empty}
Let $\ov v$ be a positive element in $H^*_{alg}(S, \ZZ)$. Then for every $H \in \Amp(S)$ and every $m \ge 1$, the moduli space $M_H(m\ov v)$ is non empty.
\end{theorem}

Let $F$ be an $H$-semistable sheaf. It is well known that $F$ admits a Jordan-H\"older filtration, which is an increasing filtration with the property that the successive quotients are $H$-stable sheaves of the same reduced Hilbert polynomial as $F$. Given $F$, this filtration depends on $H$ whereas the direct sum of the graded pieces, which will be denoted by $\gr_H(F)$ and which is an $H$-polystable sheaf, is uniquely determined by $H$.
Recall that two $H$-semi-stable sheaves $F$ and $F'$ are $S_H$-equivalent (and we write $F \sim_H F'$) if their Jordan-Holder filtration (with respect to $H$) have isomorphic graded pieces. In symbols
\[
F \sim_H F' \quad \iff \quad \gr_H(F)=\gr_H(F').
\]
The moduli space $M_H(v)$ parametrizes $S_H$-equivalence classes of $H$-semi-stable sheaves with Mukai vector $v$, and since for any $S_H$-equivalence class there is a unique $H$-polystable sheaf, we can say that $M_H(v)$ parametrizes isomorphism classes of $H$-polystable sheaves with Mukai vector $v$.
Notice also that if $F$ is $H$-stable then its $S_H$-equivalence class coincides with its isomorphism class.

From the above discussion about tangent and obstruction spaces, it follows that the singular locus of the moduli space lies in the strictly semi-stable locus $M_H(v) \setminus M^s_H(v)$ or, equivalently, in the locus parametrizing polystable sheaves with non-trivial automorphism group. There are two sources of strictly semi-stable sheaves
\begin{enumerate}
\item [1)] The Mukai vector $v$ is not primitive, i.e., $v=m \ov v$, with $m \ge 2$ and some $\ov v \in H^*_{alg}(S, \ZZ)$;
\item [2)] The Mukai vector $v$ is primitive, but the polarization $H$ is not $v$--general (see Theorem--Definition \ref{v walls} below for the definition of $v$--general polarization).
\end{enumerate}

Let us comment on these two points. Regarding item $1)$, we already said that Kaledin, Lehn, and Sorger \cite{Kaledin-Lehn-Sorger06} showed that no other example, beyond those studied by O'Grady, admit a sympectic resolution.

As for item $2)$, it is at the center of the present note. This case is quite different in nature, in that by changing the stability parameter one can always find a symplectic resolution.
For technical reasons which we will explain at the end of this section, we will from now on concentrate on the case of pure dimension one sheaves. As mentioned earlier we believe that the correct context for handling the general case is that of Bridgeland stability conditions; this will be the subject of a forthcoming paper.

\begin{defin}
A sheaf $F$ on $S$ is called pure of dimension one if its support has dimension one, and if the same holds for any non zero sub-sheaf of $F$.
\end{defin}

This means that $F$ can have $1$-dimensional, but not $0$-dimensional torsion.
If $F$ is a pure dimension one sheaf then its Fitting support, which is one-dimensional by definition, is a representative of its first Chern class. By definition, the Mukai vector of a pure dimension one sheaf $F$ of the form
\[
v(F)=(0, c_1(F), \chi(F)),
\]
and is positive in the sense of Definition \ref{positive v}, as soon as $\chi(F) \neq 0$. Let $g$ be the arithmetic genus of the Fitting support of $F$. Since $v^2=c_1(F)^2=2g-2$, it follows that $M_H(v)$ has dimension $2g$. In fact, there is a natural support morphism from $M_H(v)$ to the linear system defined by $c_1(F)$ (which is $g$-dimensional since we are on a K3 surface) that realizes this moduli space as a relative compactified Jacobian of the linear system. Since this morphism will not  play a role in the rest of the paper, we will not say anything more about it.

For a sheaf $F$ of pure dimension one, Gieseker semi-stability with respect to an ample line bundle $H$ is expressed by means of the slope
\[
\mu_H(F):=\frac{\chi(F)}{c_1(F) \cdot H}.
\]
From this one sees directly that, if the Fitting support $D$ is reduced and irreducible, then $F$ is stable with respect to \emph{any} polarization. In general, the stability of $F$ with respect to $H$ is determined by the quotient sheaves supported on the sub-curves of $D$ (for e.g., cf. Lemma 3.2 in \cite{Arbarello-Sacca-Ferretti}).

\begin{theorem-defin} [Yoshioka \cite{Yoshioka}, Huybrechts-Lehn \cite{Huybrechts-Lehn}] \label{v walls} Let $v \in H^*_{alg}(S, \ZZ)$ be a positive Mukai vector. There is a countable set of real codimension one linear subspaces in $\Amp(S) \otimes_\ZZ \mathbb{R}$ (called the \emph{walls} associated to $v$) such that if $H$ lies in the complement of these subspaces then there are no strictly $H$-semistable sheaves with Mukai vector $v$, while if $H$ lies on one of these walls, then there are strictly $H$-semi-stable sheaves with Mukai vector $v$. A connected component of the complement of the walls is called a \emph{chamber}. 
Let $\mathcal W_1, \dots, \mathcal W_k$ be a (possibly empty) set of walls. For $H$ varying in $\mathcal W_1 \cap \dots \cap \mathcal W_k$ but not on any other wall, the moduli space $M_H(v)$ is independent of $H$
 (a connected component of the set of such polarizations is called a \emph{face}; in particular a chamber is a face).
The set of walls is locally finite and, in the case where $v$ is the Mukai vector of a pure dimension one sheaf, it is actually finite.
\end{theorem-defin}

If $H$ and $H'$ are polarizations, we say that $H$ is \emph{adjacent} to $H'$, if $H'$ lies in the closure of the face containing $H$.

In the case $v$ is the Mukai vector of a pure dimension one sheaf, it is fairly straightforward to describe the walls associated to $v$.

\begin{prop}\label{walls and strata} 
Let $v=(0, D, \chi)$ be a positive Mukai vector (i.e. $D$ is an effective curve, and $\chi \neq 0$).
\begin{itemize}
\item [{\rm1)}] The walls associated to $v$ (briefly, the $v$-walls) are described by equations of the form
\[
\chi \, (\Gamma\cdot x)=\chi_\Gamma  (D \cdot x )
\]
where $\Gamma \subset D$ is a sub curve and $\chi_\Gamma $ ranges in a finite set of integers determined by $v$ and by $\Gamma$.
\item [{\rm2)}]  Let $H_0$ be a polarization that is not $v$-generic. Then, there exists a natural stratification of the singular locus of $M_{H_0}(v)$, whose strata are in one-to-one correspondence with decompositions
\[
v=\sum m_j w^{(j)},
\]
where $m_j >0$ and where $w^{(j)}$ are rank zero positive Mukai vectors.

\item [{\rm 3)}]  For any $H$ adjacent to $H_0$, there exists a morphism (cf. \cite{Zowislok}) 
\[
\begin{aligned}
h: M_H(v) & \longrightarrow M_{H_0}(v) \\
F & \longmapsto \gr_{H_0}(F)
\end{aligned}
\] 
which associates to each $H$-semistable sheaf $F$ the polystable sheaf $\gr_{H_0}(F)$ and which is an isomorphism over the locus of $H_0$-stable sheaves. In particular, if the general member of the linear system is an integral curve, then $h$ is birational.
\end{itemize}
\end{prop}
\begin{proof}
We start with the proof of $(1)$.
 Let $H$ be a polarization lying on a wall, and let $F$ be a strictly $H$-semistable sheaf with Mukai vector $v$. This means that for every quotient $F \to G$, with Fitting support equal to some sub curve $\Gamma \subset D$ we have
\be \label{inequalities}
\frac{\chi}{D \cdot H} \le \frac{\chi(G)}{\Gamma\cdot H},
\ee
and that equality holds for a least one quotient sheaf $G$. It follows that a necessary condition for $H$ to lie on a $v$-wall is that there exist a subcurve $\Gamma$ such that the rational number
\[
\chi_\Gamma:=\frac{\chi}{D \cdot H} (\Gamma\cdot H)
\]
is an integer. Conversely, if the rational number $\chi_\Gamma$ is an integer, we can exhibit a strictly $H$-semi-stable sheaf $F$ with Mukai vector $v$ in the following way. Let $\Gamma' \subset D$ be the complementary sub-curve. Since $\chi_\Gamma$ is an integer if and only if $\chi_{\Gamma'}=\frac{\chi}{D \cdot H} (\Gamma'\cdot H)$ is an integer, we only need to produce $H$-semistable sheaves $G$ and $G'$ with Mukai vectors $w=(0, \Gamma, \chi_{\Gamma})$ and $w'=(0, \Gamma', \chi_{\Gamma'})$, respectively. Indeed, then we can simply set $F=G \oplus G'$.  This can be achieved thanks to Theorem \ref{Yoshioka non empty} which guarantees that $M_H(w)$ and $M_H(w')$ are both non-empty.

The stratification in $(2)$ is defined in terms of the \emph{type} of a polystable sheaf, in the following sense. For $j=1,\dots, s$, let $m_j$ be a positive integer and $w_j$ a positive, rank-zero, Mukai vector.
An $H_0$-polystable sheaf $F$, is said to be of type
\be \label{type of poly sheaf}
\tau=(m_1, w_1; \cdots;m_s, w_s)
\ee
 if it is of the form $\oplus_{i=1}^s F_i^{m_i}$, where the $F_i$'s are distinct $H_0$-stable sheaves of Mukai vector $w_i$. Notice that each stratum is isomorphic to an open subset of (a finite quotient of) the product
\[
M_{H_0}(w_1) \times \cdots \times M_{H_0}(w_s).
\]
As for $(3)$, we argue as follows. Let $\mathfrak F$ be the face containing $H$, and let $F$ be an $H$-semi-stable sheaf with Mukai vector $v$. We need to show that $F$  is $H_0$-semistable. Since $F$ is $H$-semistable, inequalities hold in (\ref{inequalities}); some are strict inequalities, whereas those corresponding to the equations of $\mathfrak F$ are equalities. Since $H_0$ is contained in the closure of $\mathfrak F$, we can move $H$ within $\mathfrak F$ until it hits its boundary at the face $\mathfrak F_0$ containing the polarization $H_0$.
Since $\mathfrak  F_0$ lies in the boundary of $\mathfrak F$, the equalities all continue to hold. As for the inequalities, they will either continue to hold strictly or, those defining $\mathfrak F_0$ in $\mathfrak F$,  will turn into equalities and hence $F$ is $H_0$-semistable. This defines the morphism $h$.
As for the statement about the birationality, it is clear in the case when $|D|$ has no fixed component (indeed $h$ is an isomorphism on the locus of sheaves with irreducible support) but, with a little more work, it can be shown in general.
\end{proof}

\begin{defin}\label{relevant} We say that a $v$--wall $\mc W$ is relevant to the sheaf $F$, if the polarizations parametrized by $\mc W$
make $F$ strictly semistable. \end{defin}

Notice that if $H$ is $v$-generic (i.e., $\mathfrak F$ is a chamber), then $h: M_H(v) \longrightarrow M_{H_0}(v)$ is a symplectic resolution, whereas in general it is only a partial resolution.
The aim of this article is to study these morphisms, locally around a point $[F] \in M_{H_0}(v)$. We will do so by means of Nakajima quiver varieties that will be introduced in Section \ref{quivers}.

Observe that item $(3)$ in Proposition \ref{walls and strata} can fail for sheaves of positive rank (where the reduced Hilbert polynomial has two coefficients, see Example \ref{failure higher rank} below), in the sense that the morphism  associated to a degeneration of the polarization can have non empty indeterminacy locus. This failure is precisely the reason for restricting to pure dimension one sheaves.

For higher rank, one needs to consider either Bridgeland stability conditions, where the analogue of the morphism $h$ is always regular, or twisted Gieseker stability as introduced by Matsuki and Wentworth \cite{Matsuski-Wentworth} (see also \cite{Zowislok}). This will be the subject of a forthcoming paper. One example where the morphism is not defined is the following.

\begin{example}\label{failure higher rank}
\emph{
Let $S$ be a K3 surface whose the Picard group is generated by two elliptic curves $e$ and $f$, with $e\cdot f=2$. Let $n \ge 0$ be a positive integer,  let $\mc I_z \subset \mc O_S$ be the ideal sheaf of a length $n$ subscheme $z$ of $S$, and set $L=f-e$. We claim that for $n \gg 0$ the rank two sheaf defined by a non-split extension
\[
0 \to \mc O_S \to E \to  L \otimes \mc I_z \to 0.
\]
has the following property: there is a chamber in $\Amp(S)$ where $E$ is stable, but there is a wall of this chamber where $E$ is unstable. First observe that for $H_0=e+f$ the sheaf $E$ is unstable since $\mu_{H_0}(E)=0$ and $\chi(E)/2 < \chi(\mc O_S)$. Second, we claim that for any $H=ae+bf$ with $b <a<3b$, the sheaf $E$ is $H$--stable.
Indeed, to check Gieseker stability of $E$ we have to compare the slope of $E$ with that of rank one subsheaves $G \subset E$. Set $\Gamma=c_1(G)$. Since for $a>b$, $\mu_H(E)>0$ we can assume that the composition $G \to E \to L \otimes \mc I_z$ is non-zero. From this it follows that $D:=L \otimes \Gamma^{-1}$ is effective and, since we can assume $G$ to be saturated in $E$ and since the extension defining $E$ is non-trivial, we can assume that $D$ is non-trivial. We therefore only have to worry about line bundles $\Gamma$ satisfying
\be \label{disugualianze su gamma} \aligned
\ff{L \cdot H}{2} \le \Gamma \cdot H & < L \cdot H.
\endaligned
\ee
Since $L \otimes \Gamma^{-1}$ is effective but $L \otimes \Gamma^{-2}$ cannot be effective, we can write $\Gamma=f-ce$, for some $c \ge 2$. It is now easy to check that for $H$ in the range above there is no $\Gamma$ satisfying (\ref{disugualianze su gamma}).}
\end{example}

\section{Quadraticity of Kuranishi families}\label{Formality}
Let $H_0$ be a polarization on $S$ and let
\be\label{fascio}
F=F_1^{n_1}\oplus\cdots\oplus F_s^{n_s}
\ee
be an $H_0$--polystable sheaf on $S$. Here the $F_i$ are the distinct ${H_0}$-stable factors of $F$.
We denote by $G$ the automorphism group of $F$
\be\label{auto}
G:=\Aut(F)\cong\overset{s}{\underset{i=1}\oplus}\GL(n_i)
\ee
Consider the functor (cf. \cite{Huybrechts-Lehn} Section 2.1.6)
\[
\Def_F: \text{Art} \longrightarrow \text{Sets}
\]
from the category of local Artinian $\CC$-algebras to the category of sets, which assigns to a local Artinian $\CC$-algebra $A$ the set $\Def_F(A)$  of equivalence classes of pairs $(F_A, \varphi)$, where $F_A$ is a flat deformation of $F$, parametrized by $A$, and $\varphi: F_A \otimes \CC \to F$ is an isomorphism. Two pairs $(F_A, \varphi)$ and $(F'_A, \varphi')$ are equivalent if there is an isomorphism $\psi: F_A \to F'_A$ such that $\varphi' \circ \psi =\varphi$.
It is well known that functor $\Def_F$ is a deformation functor (in the sense that it satisfies conditions $H_1$ and $H_2$ of \cite{Schlessinger}). Its tangent space
is canonically identified with
\[
\Ext^1(F,F),
\]
whereas the obstruction space with
\[
\Ext^2(F,F)_0:=\ker[\tr: \Ext^2(F,F) \to H^2(\mc O_S)].
\]
By using the definition of obstruction space, one gets the so-called \emph{Kuranishi map}
\be \label{formal kuranishi}
{\kappa}=\kappa_2+\kappa_3+\cdots:  \widehat{Ext^1(F,F)} \longrightarrow \Ext^2(F,F)_0
\ee
with values in the obstruction space, which is a formal map, defined inductively on the order, having the property that the formal scheme
\be\label{formal kuranishi-2} 
D_{\kappa}:=\kappa^{-1}(0)
\ee
parametrizes a formal deformation $(\wh{ \mathcal F}, \wh \varphi)$ of $F$.  This means that, if $A$ denotes the local Artinian $k$-algebra defined by $D_{\kappa}=\Spec A$ and $\mk m \subset A$ is the maximal ideal, then $(\wh{ \mathcal F}, \wh \varphi)$ is a collection of compatible families $\{ (\mathcal {F}_n, \varphi_n) \in \Def_F(A/\mk m^n) \}$.  This family, called formal Kuranishi or versal family, has the following versal property:

\emph{for any local Artinian $k$-algebra $B$ and any equivalence class $(F_B, \varphi)$ in $\Def_F(B)$ there is a morphism $\spec B \to D_{\kappa}$ inducing $(F_B, \varphi)$ by pull-back. This morphism is not unique, but the induced tangent map is unique.}

This property determines $ D_{\kappa}$ uniquely, but not up to unique isomorphism.  The formal scheme $ D_{\kappa}$ is called the {\it versal deformation space} (or hull) by Schlessinger \cite{Schlessinger} and Rim \cite{Rim}, and miniversal deformation space by Hartshorne. The versality property translates into the fact that the second order term (but not the higher order ones) of the Kuranishi map is uniquely determined. More specifically, it can be shown \cite{Kaledin-Lehn-Sorger06} that this term coincides with the cup product map, i.e.,
\be \aligned \label{k2 cup}
\kappa_2: \Ext^1(F,F) & \longrightarrow \Ext^2(F,F)_0 \\
e & \longmapsto \kappa_2(e)=e\cup e
\endaligned
\ee

A way to construct  Kuranishi maps and versal deformation spaces is within the framework   of \emph{differential graded Lie algebras} (dgla for short). The advantage of this point of view is that it allows, in some cases, to see some properties of the Kuranishi map that cannot be seen otherwise. Given a dgla there is an abstract way of assigning to it a deformation functor (cf. \cite{Goldman-Millson},  \cite{Manetti99}, \cite{Manetti04}) and this deformation functor admits a formal versal deformation space, defined by an equation in the first graded piece of the graded algebra.  The quadratic term of the equation is canonically identified with the Lie bracket, or commutator, on the first graded piece of the graded algebra. Two observations are important for the following. First, if we start from the dgla $R\Hom(F,F)$, then the deformation functor is exactly the deformation functor $\Def_F$ defined above (cf. \cite{Manetti99})  and the versal deformation space can be identified with the base of a formal Kuranishi map (with the equation corresponding to the formal Kuranishi map defined above). Second, if the dgla has trivial differential, then the equation defining the versal deformation space is quadratic \cite{Goldman-Millson}. In particular, if this is the case then the formal deformation space can be defined by a quadratic equations, i.e. referring to (\ref{k2 cup}), we have
\be\label{def-form}
\Def_F\cong\kappa_2^{-1}(0)
\ee
The crucial observation (\cite{Goldman-Millson}, \cite{Manetti99}) is that given two quasi-isomorphic dgla's, the versal deformation spaces associated to them are isomorphic.
Recall that a dgla
 $L$
    is {\it formal} if there exists a pair of quasi-isomorphisms  of
dgla's: $L\leftarrow M\rightarrow H$,
with $H$ having trivial differential.   

\begin{defin} \label{formality}
We say that a sheaf $F$ satisfies the {\it dgla-formality condition}
if the dgla   $R\Hom(F,F)$ is formal. We say that $F$ satisfies the quadraticity property if the deformation space is a complete intersection of quadrics, i.e. if (\ref{def-form}) holds. The formality property implies the quadraticity property.

\end{defin}

In \cite{Kaledin-Lehn07}, Kaledin and Lehn prove the following proposition

\begin{theorem}[\cite{Kaledin-Lehn07}] Let $S$ be a K3 surface, and let $\mc I_z$ be the ideal sheaf of a subscheme $z \subset S$ of finite length. Then the polystable sheaf $E= \mc I_z^{\oplus n}$ satisfies the the formality property, i.e., the dgla $R\Hom(E,E)$ is formal.
\end{theorem}

Inspired by Kaledin Lehn's work, Zhang  \cite{Zhang12} proves the following theorem.

\begin{theorem} \label{Zhang} Let $(S,H)$ be a polarized K3 surface. Let $v_0$ be a primitive Mukai vector of positive rank and such that  there is at least one $\mu_H$-stable sheaf on $S$ with Mukai vector $v_0$. Let $m$ be a positive and let $E$ be an $H$--polystable sheaf with Mukai vector $v=mv_0$ whose decomposition in non isomorphic stable summands is given by
$$
E= \overset{s}{\underset{i=1}
\oplus}E_i^{n_i}
$$
Suppose that $v(E_i) \in \mathbb N v_0$ and, for every $i$, let $r_i$ be the rank of $E_i$. Then $E$ satisfies the dgla-formality property in following cases:
\begin{enumerate}
\item[(1)] when $r_i\geq 2$ for all $i$;\\
\item[(2)] when $r_i=1$ for all $i$.
\end{enumerate}
\end{theorem}

The proof builds on the work of Verbitsky \cite{Verbitsky}, who introduced the notion of hyperholomorphic bundles and  of Kaledin \cite{Kaledin07}, which gives certain criteria for when formality holds in families (see also \cite{Lunts}).
The proof of part $(1)$ uses the Ulhenbeck--Yau theorem on the existence of a Hermitian--Einstein connection on stable vector bundles, which under the assumptions of the theorem guarantees a Hermitian--Einstein connection on $E$, and therefore on $E \otimes E^\vee$. This allows one to conclude that $E\otimes E^\vee$ is a hyperholomorphic sheaf which, in turn, allows one to use Theorem 4.3 of \cite{Kaledin07}. Recall that given a hyperk\"ahler metric on a K3 surface $S$, there is a whole $\PP^1$ of complex structures for which that metric stays  hyperk\"ahler.  More precisely, one defines a {\it twistor family} $\X\to \PP^1$, whose total space is diffeomorphic to $X\times\PP^1$ and whose fibres are copies of $X$ equipped with the complex structure parametrised by $\PP^1$.
Roughly speaking, a hyperholomorphic sheaf is a sheaf $\F$ on $\X$ that is holomorphic with respect to all of these complex structures.
Since for the general complex structure parametrized by this $\PP^1$, the corresponding K\"ahler surface has no holomorphic curves, a necessary condition for a sheaf to be hyperholomorphic is that its first Chern class is trivial (for a partial converse see Theorem 3.9 of \cite{Verbitsky}). For this reason one cannot use this strategy to prove the a formality result in the case of pure dimension one sheaves.

\begin{rem}\label{remark zhang} 1) A first remark  about Zhang's paper is the following. It is not immediately apparent that the hyperholomorphic sheaf $\F$ extending $E\otimes E^\vee$ carries an algebra structure. The author explained to us how to proceed.
The algebra structure on $F=E^\vee \otimes E$, is given by a contraction map  $F \otimes F \to F$, that is by a
 global section of the sheaf $G = F^\vee \otimes F^\vee \otimes F$ on $X$. To defined the algebra structure it is enough to extend this section to a global section of
 $ \mathcal{G} = \mathcal{F}^\vee \otimes \mathcal{F}^\vee \otimes \mathcal{F}$. Now $\G$ is a hyperholomorphic sheaves
 and one may use Proposition 3.4 \cite {Zhang12} (i.e. Proposition 6.3 in  \cite{Verbitsky}) on $G$ and $\mathcal{G}$, for $i=0$.

2) A second important remark about Zhang's paper is the following.
 Looking into the proof of Zhang's theorem, one sees that the assumption on $H$ and $E$ is not  necessary since for the existence of a Hermitian--Einstein connection on $E$ one only needs to assume that the ratio $(c_1(E_i) \cdot H)\slash r_i$ is independent of $i$, a condition which is satisfied by assumption since $E$ is polystable.
It follows that one can state the theorem also in the case where the polystablity of $E$ comes not from the non--primitiveness of the Mukai vector, but from the fact that the polarization is not general.

\end{rem}

We will now describe  a method to reduce the problem of quadraticity of the deformation space $\Def_F$ for a pure dimension one sheaf $F$, to the case of formality for  positive rank sheaves where one can use Zhang's result.
A first example of this procedure is given by Lazarsfeld--Mukai bundles.

As usual, let $S$ denote a K3 surface.
We will say that a pure,  dimension-one sheaf $F$ on a K3 surface $S$,  is {\it non-special}
if the following conditions are satisfied.
\be\label{conditions_F}
\aligned
&a)\,\,\, F\,\,\, \text{ is generated by its sections,}\\
&b)\,\,\, H^1(S, F)=0\,.
\endaligned
\ee

The kernel $M_F$ of the evaluation of global sections of $F$,defined by the exact sequence
\be\label{LM_start_F1}
0\to M_F\to H^0(S, F)\otimes {\mc O}_S\to F\to 0,
\ee
is locally free and its dual 
$$
E_F=M_F^\vee
$$
is called the Lazarsfeld-Mukai sheaf associated to $F$.

For the first properties
of these bundles see \cite{Lazarsfeld-Petri}. Taking $\Shom(\,\,\,,{\mc O}_S)$ of (\ref{LM_start_F1})
one easily establishes the following equalities
\be
\label{basic-eq}
 h^1(S,M_F)= h^1(S,E_F)=h^2(S,E_F)=h^0(S,M_F)=0
\ee
and since we are assuming  $H^1(S, F)=0$, from the dual of (\ref{LM_start_F1}) we also get an isomorphism
 \[
H^0(S, F)^\vee\cong H^0(S, E_F)
 \]
 The following two facts can be easily verified directly.
 
 {\bf Fact 1.} {\it There is an  isomorphism of differential graded Lie algebras:} 
$$
\Ext^\bullet(F, F) \cong \Ext^\bullet(M_F, M_F)=\Ext^\bullet(E_F, E_F).
$$ 

\vskip 0.3 cm
{\bf Fact 2.} {\it Let $G=\Aut(F)$. Then $G\cong \Aut(M_F)$ and there is a $G$-equivariant isomorphism of functors}
\[
\eta: \Defo_F\to \Defo_{M_F}.
\]

Putting together the two facts above
we get the following result.

\begin{prop} \label{defm and deff} Let $S$ be a K3 surface. Let $F$ be a non-special pure dimension one sheaf and let $M_F$ be its Lazarsfeld-Mukai  bundle.  There exists a $G$--equivariant isomorphism between $\Def_F$ and $\Def_M$. In particular, $\Def_F$ enjoys the quadraticity property if and only if the  $\Def_{M_F}$ does, too.
\end{prop}

From Zhang's Theorem and Remark \ref{remark zhang} we then get the following quadraticity criterion in the pure dimension one case.

 \begin{prop}\label{rank0form}Let $(S,H_0)$ be a polarized K3 surface. Let $v = (0,[C],\chi)$ be a Mukai vector and let $F$ be a  $H_0$-polystable sheaf on $S$, pure of dimension one and non-special, and with Mukai vector equal to $v$. If $M_F$ is $H_0$-polystable, then both $\Def_{M_F}$ and
$\Def_F$ satisfy the quadraticity property.
\end{prop}

The main result of this section is the following theorem.

\begin{theorem} \label{formality-2} Let $x=[F]\in M_{H_0}(v)$ be a point corresponding to a $H_0$-polystable sheaf $F$  pure of dimension one on $S$. Then the deformation space of $F$ satisfies the quadraticity property.
\end{theorem}

In view of Proposition \ref{rank0form}, in order to prove this theorem it suffices to reduce ourselves to the case where $F$
is non-special and then prove that $M_F$ is $H_0$--polystable. The first task is easily fulfilled.
In fact, tensoring $F$ with any power of $H_0$, preserves $H_0$-(poly)stability of $F$ and gives an isomorphism of $M_{H_0}(v)$ onto $M_{H_0}(v_n)$, where $v_n=v(F(nH_0))$. Using Lemma \ref{H to H'} we may assume that $F$ is non-special. 

We are thus reduced to proving the following theorem. In proving this theorem, we will appeal to results by Yoshioka that were kindly pointed out to us by the author himself.

\begin{theorem}\label{stabil-LM} Let $F=\oplus F_i^{n_i}$ be a non-special, pure, dimension one sheaf on $S$ which is polystable with respect to a given polarization $H_0$. Then $M_F$ is $ H_0$-polystable.
\end{theorem}

\begin{proof}
  More precisely we will make use of Proposition 1.5 and Corollary 2.14 in Yoshioka's paper \cite{Yoshioka-FM}. When possible, will also adopt  the notation of that paper.

Consider the Mukai vector $v_0=(1,0,0)$.  A sheaf $E$ on $S$ such that 
$v(E)=v_0$,  is of the form $E=\I_p$, for some $p\in S$. Moreover,
every polarisation $H$ is $v_0$-generic. As in Theorem 1.7 of \cite{Yoshioka-FM}, we set
$$
Y=M_H(v_0)\cong S
$$

We then let $\E=\I_\Delta\subset S\times Y$ and consider the diagram
$$
\xymatrix{&\I_\Delta\subset S\times Y\ar[dl]_{p_S}\ar[dr]^{p_Y}\\
S&&Y
}
$$
and the  Fourier-Mukai equivalence  attached to $\E$:
$$
\begin{aligned}
\F_\E: \mathcal D^b(S) & \longrightarrow \mathcal D^b(Y) \\
G & \longmapsto  R {p_Y}_*  (\mathcal{E} \otimes p_S^* G  )
\end{aligned}
$$ 
Set

$$
G_1=\E^\vee_{|S\times\{y\}}=\mc O_S\,,\qquad G_2=\E_{|\{x\}\times Y}=\I_x\subset \mc O_Y
$$
Let $v$ be a Mukai vector. We should think of $v$ as $v=v(F)$ or as $v=v_i=v(F_i)$, accordingly
$$
v=(0, D, \chi)\,,\quad\text{or}\quad v=(0, D_i, \chi_i)\,,\quad\text{with}\quad 
\chi\neq0\,,\quad \chi_i\neq0
$$

We can write
$$
v=(0, D, \chi)=l(1,0,0)+a(0,0,1)+ dH_0+\ov D,
$$
where
$$
\ov D\in (\operatorname{NS}(S)\otimes\QQ)\cap H_0^\perp, \quad dH_0+\ov D=D, \quad d\in\frac{1}{H^2}\ZZ
$$
and $l=0$, $a=\chi$. We are thus in case 2) of Theorem 1.7 in \cite{Yoshioka-FM},
so we must be sure that
\be\label{condition-a}
a>\max\left\{3, \frac{d^2H_0^2}{2}+1\right\}
\ee
Now $d$ (or the $d_i$'s) only depend on $D$ (or on the $D_i$'s) and this is a finite set of numbers.
Twisting $F$ and the $F_i$'s by $nH_0$ does not change the $H_0$-stability of the  $F_i$'s or the polystability of $F$
but allows one to increase at will the value of $a=\chi$ (or $a=\chi_i)$ insuring the validity of (\ref{condition-a}).
From Theorem 1.7 \cite{Yoshioka-FM} we infer that $\F_\E$ induces an isomorphism
$$
M_{H_0}(v)^{ss}\overset\cong\longrightarrow M^{G_2}_{\wh H_0}(\F_\E(v))^{ss}\,,\qquad (G_1=\mc O_S)
$$
which preserves $S$--equivalence classes.
We must then identify the polarisation $\wh H_0$, the Mukai vector $\F_\E(v)$ and address the question of $G_2$-twisted-$\wh H_0$-semistability. Since $F$ and the $F_i$'s are non-special we have
$$
\F_\E(F)=M_{F}\,,\qquad \F_\E(F_i)=M_{F_i}
$$
so that
$$
\F_\E(v)=(\chi, -D, 0)\,,\qquad \F_\E(v_i)=(\chi_i, -D_i, 0).
$$

Moreover,  
$$
0\to\F_\E(H_0)\to H^0(S,\mc O(H_0))\otimes\mc O_S \to \mc O_S(H_0)\to 0
$$
so that, by  formula (1.4) in \cite{Yoshioka-FM}, we get
$$
\wh H_0=-[\F_\E(H_0)]_{_{1}}=H_0
$$

The conclusion is  that, if $F_i$ is $H_0$-stable, then $M_{F_i}$ is 
$G_2$-twisted-$ H_0$-stable.  It remains to prove that $M_{F_i}$ is in fact
$\mu_{H_0}$-stable. For this we use
 Corollary 2.14 in \cite{Yoshioka-FM}. To put ourselves in the hypotheses of that corollary,
we must check that $M_{F_i}$ is $\mu_{H_0}$-semistable. Set $G=G_2(=\I_x)$ and $M=M_{F_i}$.
By definition of $G$--twisted stability we have
$$
\frac{\chi_G(N(nH_0))}{\rk_G(N)}< \frac{\chi_G(M(nH_0))}{\rk_G(M)}
,\quad 0\subsetneq N\subsetneq M\,,\quad{and}\,\quad n>>0\,,
$$
where
$$
\chi_G(x)=v(G)^\vee \cdot v( x)\,,\quad \text{and}\quad \rk_G(x)=[\ch(G)^\vee\cdot \ch (x)]_0
$$

Since
$$
v(M)=(\chi_i,\, -D_i, \,0)
$$
we have
$$
\frac{\chi_G(M(nH_0))}{\rk_G(M)}=\frac{n^2H_0^2}{2}-n\frac{D_i\cdot H}{\chi_i}
$$
Suppose $N=(r,s,t)$ with $s=-\Gamma$, then
$$
\chi_G(N(nH_0))=\left(1,\, nH_0,\,\frac{n^2H_0^2}{2}\right)\cdot (r,\,-\Gamma,\,t)
$$
Thus

$$
\frac{\chi_GN(nH_0))}{r}=\frac{n^2H_0^2}{2}-n\frac{\Gamma\cdot H_0}{r}+t\,.
$$
Now, $G$-twisted-$H_0$-stability means  that, for $n>>0$,
$$
t-n\frac{\Gamma\cdot H_0}{r}+\frac{n^2H_0^2}{2}
<-n\frac{D_i\cdot H}{\chi_i}+\frac{n^2H_0^2}{2}
$$
i.e.
$$
t+n\mu_{H_0}(N)
<n\mu_{H_0}(M)\,\qquad\text{for}\quad n>>0
$$
This shows that 
$\mu_{H_0}(N)
\leq\mu_{H_0}(M)$,
proving the semistability of $M$. This ends the proof of Theorem \ref{stabil-LM}
and thus also of Theorem \ref{formality-2}.
\end{proof}

\section{$G$-equivariant Kuranishi families}\label{G-equi-kuran}

Let us consider again  the Kuranishi map (\ref{formal kuranishi}).
It is important to remark that since the Kuranishi map is not unique, there is no a priori reason for it to be $G$-equivariant (with respect to the natural action of $G$ on the $\Ext$-groups), nor for there to be a natural action of $G$ on the base of a versal family. However, with some additional work, one can construct a $G$-equivariant map formal Kuranishi map $\kappa$  and a $G$-equivariant formal family $ (\wh {\mathcal F}_\kappa,  \wh \varphi_\kappa)=\{ (\mathcal {F}_n, \varphi_n) \in \Def_F(A/\mk m^n) \}$ parametrized by  the corresponding $D_{\kappa}$; roughly speaking, this means that for every $n$ the sheaf $\mathcal {F}_n $ is $G$-linearized with respect to the action of $G$ on $ {\kappa}^{-1}(0)_n=\Spec A/\mk m^n $. The reader may look at Rim's paper \cite{Rim} for a more detailed discussion. The main result of loc. cit. is the following theorem

\begin{theorem}[Rim] \label{rim} Let the notation be as above.
A $G$-equivariant formal Kuranishi map $\kappa$ and a $G$-equivariant formal family $ (\wh {\mathcal F}_\kappa,  \wh \varphi_\kappa)$ on ${ \kappa}^{-1}(0)$  exist and are unique up to unique $G$-equivariant isomorphism. \end{theorem}

The abstract nature of this theorem makes it often hard to compute, in practice, $G$--equivariant maps and families.
A very nice and explicit construction of a $G$-equivariant formal Kuranishi map (even though not of a $G$--equivariant family) is given in Appendix A of \cite{Lehn-Sorger06}; as for $G$-equivariant families see Proposition \ref{etale slice} below.

The next important feature of the Kuranishi map and of the Kuranishi family comes in relation to the Quot scheme.
Let us start with some notation.
We denote by $\Quot_{H_0}$ the irreducible component of an appropriate Quot scheme, constructed by using the polarization $H_0$,  containing the $H_0$-polystable sheaf $F$, and we let  $\GL(\VV)$ be the natural group acting on it (the appropriate Quot scheme needed for the proof of the main theorem will be specified in Section \ref{remarks stability criteria}). Let  $\Quot^{ss}_{H_0}$ be the open subset parametrizing $H_0$-semistable sheaves and let us fix a point $q_0 \in \Quot^{ss}_{H_0}$, corresponding to the sheaf $F$.  It is well known that $\Stab(q_0)\cong G$, so that the point $q_0$ has reductive stabilizer. One can therefore consider an {\it \'etale slice }
\be \label{Q slice}
Z \subset \Quot_{H_0}^{ss}
\ee
at $q_0$ for the action of $\GL(\VV)$ on $\Quot_{H_0}^{ss}$ (cf. \cite{Drezet-Luna}).  By definition, the \'etale slice $Z$ is a locally closed and $G$-invariant  affine subvariety of $\Quot_{H_0}^{ss}$ containing $q_0$,  having the property that the natural morphism
\be \label{etale morphism}
\epsilon: Z \sslash G \to \Quot_{H_0}^{ss} \sslash \GL(\VV)=M_{H_0},
\ee
is \'etale (for a more precise statement see, for example, Theorem 5.3 of \cite{Drezet-Luna}). For later use we now present a brief sketch of the construction of $Z$.

It is known that since $F$ is an $H_0$-semistable point, one can use the natural $\GL(\VV)$-linearized ample line bundle on $\Quot_{H_0}$ (cf. \cite{Huybrechts-Lehn}, Section 4.3), to define a $\GL(\VV)$-equivariant embedding of an affine open neighborhood of $q_0 \in \Quot_{H_0}^{ss}$ into an affine space $\mathbb A^N$, acted on linearly by $\GL(\VV)$, so that the $G$-fixed point $q_0\in \Quot_{H_0}^{ss}$ is mapped to the origin in $\mathbb A^N$. Recall that the tangent space to the $\GL(\VV)$-orbit of $q_0$ is canonically identified with $\Ext^1(F,F)$ and that the affine space $\mathbb A^N$ can be identified with its tangent space $T_0 \mathbb A^N$  at the origin. Consider the natural embedding $
\Ext^1(F,F) \subset T_0 \mathbb A^N$, and let $t: \Ext^1(F,F) \to T_0 \mathbb A^N \cong \mathbb A^N$ be the composition.
One can check that
 \[
Z= t(\Ext^1(F,F)) \cap \Quot_{H_0}^{ss},
\]
satisfies (\ref{etale morphism}) above. We would like to think of $Z$ as sitting $G$--equivariantly inside $\Ext^1(F, F)$. This is possible locally:  if $s:T_0 \mathbb A^N \to \Ext^1(F,F)$ is a $G$-equivariant splitting of the natural inclusion, then the composition
\[
Z \hookrightarrow \mathbb A^N\cong T_0 \mathbb A^N \stackrel{s}{\to} \Ext^1(F,F)
\]
is surjective at the level of tangent spaces and hence is \'etale onto its image. This means that if we look at the completion of the \'etale slice $\wh Z$ at the point $q_0$, then we can think of it as sitting \emph{$G$-equivariantly} inside $\wh{\Ext^1(F,F)}$.
Over $\Quot_{H_0}^{ss}\times S$ there is a universal family $\wt\F$ of sheaves and we
consider its restriction to $Z \times S$:
\be \label{family on slice}
\F=\wt\F_{|Z \times S}
\ee
By construction, the action of $G$ on $Z$ is {\it modular}, i.e.,  for $q\in Z$ and $g\in G$ the sheaf $\F_q$ is isomorphic to the sheaf
$\F_{gq}$ and, moreover, since it is the restriction of $\wt\F$, the family $\mc F$ is $G$-linearized (this will play an important role in Section \ref{linear}).
Notice that since $\wh Z \subset  \wh{\Ext^1(F,F)}$ parametrizes a formal family we can construct, again using the definition of the obstruction space, a formal Kuranishi map ${ \kappa}_Z:  \wh{\Ext^1(F,F)}\longrightarrow \Ext^2(F,F)_0$ such that
\be \label{Z in Ext}
\wh Z \cong {\kappa}_Z^{-1}(0).
\ee
Notice, also that since this family is $G$-linearized it is also $G$-equivariant in the sense of Rim.
We sum up these results in the following proposition

\begin{prop} \label{etale slice}
Let $\kappa:  \wh{\Ext^1(F,F)} \rightarrow \Ext^2(F,F)_0 $ be a formal $G$--equivariant Kuranishi map for the polystable sheaf $F$ and let $\wh {\F}$ be a $G$--equivariant versal family parametrized by ${\kappa}^{-1}(0)$ (which exist by Rim's theorem). Then the local completion $\wh Z$ of $Z$ at $q_0$ is isomorphic to ${\kappa}^{-1}(0)$ and, under this isomorphism, the universal family $\mc F$ induces $\wh {\F}$. Moreover, there is a unique $G$--equivariant isomorphism $\wh Z \cong {\kappa}^{-1}(0)$. In particular, ${\kappa}^{-1}(0)$ and $\wh {\F}$ are $G$--equivariantly algebraizable.
\end{prop}

When the dgla-formality property holds for $F$, we can always find a Kuranishi map so that the corresponding base of the versal family for $F$ is the completion at the origin of a complete intersection of quadrics in affine space. In particular,
the base of the Kuranishi family coincides with its tangent cone and the Kuranishi map is automatically $G$-equivariant.

\begin{rem}
Consider a polystable sheaf $F$ enjoying the dgla-formality property.
We then have at our disposal two Kuranishi families. The base, $\kappa_2^{-1}(0)$, of one family
is a complete intersection of quadrics, it is  naturally acted on by $G$, but the family parametrized by it  has no a priori natural $G$-linearization.
The base of the second family is an analytic neighbourhood of a point in an \'etale slice $Z\subset \Quot$
and the family of sheaves parametrized by it has a natural $G$-linearization. The advantage of the first family is the simplicity of its base, while the advantage of the second is its $G$-linearization.
\end{rem}

The central result in this section is the following proposition that reconciles these two advantages. 
 
\begin{prop}
 \label{Q-is-a-cone}Let $F$ be a polystable sheaf as above. Assume that  $F$ satisfies the  dgla-formality condition. Let then  $\kappa=\kappa_2$ be the quadratic Kuranishi map.
 Let $q_0\in \Quot_{H_0}^{ss}$ be the point corresponding to $F$
 and let $(Z,q_0)$ be an \'etale slice through $q_0$.
Then, there is a $G$-equivariant local analytic isomorphism,
\[
\psi:(Z,q_0) \cong (k_2^{-1}(0), 0).
\]
inducing the identity on tangent spaces: $T_{q_0}(Z)=\Ext^1(F,F)=T_0(k_2^{-1}(0))$.
In particular, there is a $G$-linearized deformation of $F$ parametrized by a $G$-equivariant analytic open neighborhood of the origin in $ k_2^{-1}(0)$.
\end{prop}

The proof of the preceding proposition is based on two results. The first one is a formal version of the proposition itself, while the second one consists in a $G$-equivariant version of Artin's approximation theorem.

Before starting with the formal result, we need some notation.
Let $(A, \mathfrak m)$  and $(B, \mathfrak n)$
be  local, complete $k$-algebras. For $s>r$, let
$$
\eta_{r,s}:A/\mathfrak m^{s+1}\to A/\mathfrak m^{r+1}\,,\qquad \zeta_{r,s}:B/\mathfrak m^{s+1}\to B/\mathfrak m^{r+1}
$$
be the natural projections.

\begin{defin}\label{notation-starr}Let $(A, \mathfrak m)$  and $(B, \mathfrak n)$
be  as above. A formal isomorphism
between $(A, \mathfrak m)$  and $(B, \mathfrak n)$ is a collection $u=\{u_r\}_{r\in \mathbb N}$ of 
compatible isomorphism $u_r: A/\mathfrak m^{r+1}\to B/\mathfrak n^{r+1}$, for $r>0$. This means that
 $\zeta_{r,s}u_s=u_r\eta_{r,s}$, for $s>r$. When this compatible system exists we say that the single isomorphisms $u_r$ extend to the formal isomorphism $u$.
\end{defin}

Consider the algebraic group of $k$-algebra automorphisms of $A/\mathfrak m^{r+1}$.
\[
\Aut_{r}(A):=\Aut_{k\text{-alg}}(A/\mathfrak m^{r+1})
\]
with the obvious projections
\[
p_{r,s}: \Aut_{s}(A)\longrightarrow \Aut_{r}(A)\,,\quad s>r.
\]

\be\label{aut-any-ord}
\mathbf{Aut}_{r}(A)=\{h_r\in \Aut_{r}(A)\,\,|\,\,h_r\,\,\text{extends to a formal automorphism}\}
\ee

\begin{defin} An action of $G$ on $(A, \mathfrak m)$ is the datum of a sequence of group homomorphism
\be
\label{g-action}
u_r : G\longrightarrow \mathbf{Aut}_{r}(A)
\ee
such that $p_{r,s}u_s=u_r$ for $s>r$.
\end{defin}

The proof of the following proposition, which is the first ingredient in the proof of Proposition 
\ref{Q-is-a-cone}, was communicated to us by Jason Starr.

\begin{prop}[J. Starr]\label{starr2}Let $(A, \mathfrak m)$  and $(B, \mathfrak n)$
be  local, complete $k$-algebras
acted on by a reductive algebraic group $G$.
Assume that there is a formal  isomorphism between $(A, \mathfrak m)$  and $(B, \mathfrak n)$ inducing a
$G$-equivariant isomorphism
from $A/\mathfrak{m}^2$ to $B/\mathfrak{n}^2$. Then there is a $G$-equivariant formal isomorphism between $(A, \mathfrak m)$  and $(B, \mathfrak n)$.
\end{prop}

\proof 

Consider the affine scheme of isomorphisms between $A/\mathfrak m^{r+1}$ and 
$B/\mathfrak n^{r+1}$:
$$
\Iso_r(A,B)=\Iso(A/\mathfrak m^{r+1}, B/\mathfrak n^{r+1})
$$
In analogy with (\ref{aut-any-ord}) we denote by $\mathbf{Iso}_r(A,B)$
the closed subscheme of $\Iso_r(A,B)$ of those isomorphisms
that extend to formal isomorphisms. By definition the projection map 
\be \label{projr}
p_r: \isob_r(A, B)\longrightarrow \isob_{r-1}(A, B)
\ee
is surjective. Since, by hypothesis, $\isob_1(A, B)$ is non-empty, so is $\isob_r(A, B)$, for every $r>1$.

The automorphism group $\autr(A)$
(resp. $\autr(B)$) acts  on the right (resp. on the left) on $\Iso_r(A,B)$.
The induced actions on $ \isob_r(A,B)$ are faithful and transitive, so that $\isob_r(A,B)$  is both an $\autr(A)$-torsor
and a $\autr(B)$-torsor. Using (\ref{g-action}), we get an induced action of  $G$ on $\Iso_r(A,B)$ and on $\isob_r(A,B)$, via conjugation.  A  fixed point in $\Iso_r(A,B)$ for this action is nothing but a $G$-equivariant isomorphism between $A/\mathfrak m^{r+1}$ and $B/\mathfrak n^{r+1}$. By Hypothesis $\isob_1(A,B)$ has a $G$-fixed point.
We must show that, for $r>1$, the projection map (\ref{projr})
is surjective on $G$-fixed points.

For $r>1$, set
$$
K_r(A)=\ker\{\Aut_{r}(A)\to \Aut_{r-1}(A)\}\,,\quad
\mathbf {K}_r(A)=\ker\{\mathbf{Aut}_{r}(A)\to \mathbf{Aut}_{r-1}(A)\}
$$
Both $K_r(A)$ and $\mathbf {K}_r(A)$ are normal abelian subgroups.
We think of them as additive groups and in fact as finite-dimensional  $k$-vector spaces. The group $G$ acts on these two vector spaces spaces via conjugation, yielding two finite dimensional linear representations of $G$.  Given a $G$-fixed point $\phi_r\in \isob_{r-1}A, B)$ the fiber $F=p_r^{-1}(\pi_{r-1})$is a $K_r(A)$-torsor with compatible $G$-action, meaning that the natural morphism $K_r(A) \times F\to F$ is $G$-equivariant. Now $K_r(A)$-torsor with compatible $G$-action are classified by $H^1(G, K_r(A))$
\footnote{To see this from a topological point of view, set $V=K_r(A)$ and consider the  $V$-torsor over $EG$ given by
$$
\wt P=EG\times F\to EG
$$
By the $G$-equivariance of $F\times V\to $F  one may form the quotient
$$
P=(EG\times F)/G
$$
and obtain a $V$-torsor over $BG$, i.e an element in  $H^1(BG, V)=H^1(G,V)$. Whenever this element is trivial the $V$-torsor $P$ is trivial and thus $\wt P$ and  $F$ must be  trivial as well. For a more algebraic argument one may substitute $EG$ with $\{*\}$ and $BG$ with the stack $\{*\}/G$,
and proceed in a similar way.
 }. Since $G$ is reductive this cohomology group vanishes so that $F$ 
must be the trivial torsor, meaning that there is a $G$-fixed point 
$\phi_r$ over $\phi_{r-1}$.
\endproof

The second ingredient for the proof of Proposition \ref{Q-is-a-cone} is the following result by Bierstone and Milman. 

\begin{prop}[Bierstone and Milman \cite{Bierstone-Milman}]  \label{BM} Let $(X, x_0)\subset \CC^n$ and 
$(Y, y_0)\subset\CC^p$ be germs of algebraic varieties acted on by a reductive group $G$. Suppose
there is a $G$-equivariant morphism from $(\wh \CC^n)_{x_0}$ to $(\wh \CC^p)_{y_0}$
inducing an isomorphism 
$\wh \phi$
between the formal neighbourhoods
 $(\wh X, x_0)$ and $(\wh Y, y_0)$. Let $c \in \mathbb N$. Then there is a local analytic $G$-equivariant isomorphism $\phi$ 
between $(X, x_0)$ and $(Y, y_0)$ which is equal to $\wh \phi$ up to order $c$.
\end{prop}

As stated above, the result, does not formally appear in \cite{Bierstone-Milman}, but it 
follows immediately from the remarks on pages 121-122 therein.
We are now ready for the proof of Proposition \ref{Q-is-a-cone}.

\begin{proof} [Proof of Proposition \ref{Q-is-a-cone}] Recall the way in which
the \'etale slice $Z$ is constructed. From the discussion just above Proposition \ref{etale slice}
 and from Rim's uniqueness theorem, we
may assume
the existence of a formal $G$-equivariant Kuranishi map $h$ such that
$$
\wh Z= h^{-1}(0)\subset \wh\Ext^1(F,F)\,.
$$
Clearly,
$$
\qquad h_2(e)=\kappa_2(e)=e\cup e
$$
Our aim is to find a $G$-equivariant formal isomorphism between $\wh Z$ and $\kappa_2^{-1}(0)$. Unfortunaley, since 
we don't know if $\kappa_2^{-1}(0)$ carries a $G$-linearized formal deformation of $F$, we can't apply Rim's uniqueness theorem. We do know that $\wh Z \cong \kappa_2^{-1}(0)$, though perhaps not $G$-equivariantly. However, letting $(A, \mathfrak m)$ and $(B, \mathfrak n)$ be the complete local $k$-algebras corresponding to $\kappa_2^{-1}(0)$ and $\wh Z$, respectively, we may apply Proposition \ref{starr2} (up to the second order the map is indeed unique, and hence $G$--equivariant) and find a $G$-equivariant formal isomorphism
$$
\alpha:\,\,\wh Z= h^{-1}(0)\overset\cong\longrightarrow \kappa_2^{-1}(0).
$$
If we can prove that $\alpha$ is induced by  a $G$-equivariant morphism $\wt\alpha: \wh\Ext^1(F,F)\to  \wh\Ext^1(F,F)$ then Proposition
\ref{Q-is-a-cone} follows at once from
from Proposition \ref{BM}. 
So let $(C, \mathfrak M)$ be the completion at $0$ of the polynomial ring
$\underset{n\geq 0}\oplus S^n\Ext^1(F,F)^\vee$.  Both $A$ and $B$ are quotients of $C$ and we denote by 
$$
\sigma_r: C/\mathfrak M^{r+1}\to A/\mathfrak m^{r+1}\,,\qquad \tau_r: C/\mathfrak M^{r+1}\to B/\mathfrak n^{r+1}
$$
the induced quotient maps.
We are interested in diagrams of type
$$
\xymatrix{ C/\mathfrak M^{r+1}\ar[r]^{\psi_r}\ar[d]_{\sigma_r}&C/\mathfrak M^{r+1}\ar[d]^{\tau_r}\\
A/\mathfrak m^{r+1}\ar[r]^{\phi_r}&B/\mathfrak n^{r+1}
}
$$
where $\phi_r$ is the  $G$-equivariant homomorphism induced by $\alpha$. 
We recall the notation introduced 
after Definition \ref{notation-starr} and we set
$$
\wt\Aut_{r}(C)=\{\psi_r \in\Aut_{r}(C)\,\,|\,\,\psi_r\,\,\text{lifts a $G$-equivariant }\,\, \phi_r\in \Iso_r(A,B)\}
$$
$$
\wt\autr(C)=\{\psi_r \in\autr(C)\,\,|\,\,\psi_r\,\,\text{lifts a $G$-equivariant }\,\, \phi_r\in \isob_r(A,B)\}
$$
By hypothesis $\wt\autone(C)$ is not empty and contains a $G$-fixed point.
The task is to show that the projection map 
$$
\wt\autr(C)\longrightarrow \wt\autb_{r-1}(C)
$$
is surjective on $G$-fixed points.
We then follow, step by step, the proof of Proposition \ref{starr2} and prove
 that if there  is a formal  $G$-equivariant isomorphism $\phi$ between $(A, \mathfrak m)$  and $(B, \mathfrak n)$ inducing the identity  $\mathfrak{m}/\mathfrak{m}^2=\mathfrak{n}/\mathfrak{n}^2$, then there is a $G$-equivariant formal automorphism $\psi$ of  $(C, \mathfrak M)$  lifting $\phi$.
\end{proof}

\section{Generalities on quiver varieties}\label{quivers}

In this section we recall a few basic facts regarding geometric invariant theory (GIT) \cite{MumfordGIT}, \cite{Dolgachev-GIT}, \cite{Dolgachev-Hu}, mostly to set up the notation, and about quiver varieties. For the latter, we loosley follow the exposition of \cite{Ginzburg}, and then present the results from \cite{Nakajima-branching}, \cite{Crawley-Boevey}, and \cite{Crawley-Boevey-norm} that will be needed in the following. For more details on the subject the reader may also consult \cite{Nakajima-Kac-Moody}. 

Let $G$ be a reductive group and let $A$ be a finitely generated $\CC$-algebra. Suppose that the affine variety $X=\spec A$ is acted on by  $G$, and consider the GIT quotient
$$
X\sslash G=\spec A^G
$$
This is a good categorical quotient.  Given a rational character $\chi:  G\to \CC^\times$, let
$$
A_n=\{f\in A\,\,|\,\,(g\cdot f)(x)=\chi(g)^nf(x)\}
$$
be the vector space of $\chi^n$--invariant functions and
and set
$$
X\sslash_\chi G=\operatorname{Proj}\left(\underset{n\geq 0}\oplus A_n\right)
$$
\begin{defin}\label{definition stable affine} A point $x\in X$ is said to be $\chi$-semistable if there exists an $f\in A_n$ such that $f(x)\neq0$, it is called stable if, in addition, the action of $G$ on $X_f$ is closed and the stabilizer of $x$ is finite. We denote by $X(\chi)$ resp. $X(\chi)^s$ the locus of  $\chi$-semistable, resp. $\chi$-stable
points in $X$. Two $\chi$-semistable points are said to be $S_\chi$-equivalent if and only if the closure of their orbits meet in $X(\chi)$. In each $S_\chi$-equivalence class there exists a unique closed orbit (which is of minimal dimension and has reductive stabilizer).
\end{defin}
By construction, he natural morphism
\be\label{GIT-quot}
X(\chi)\longrightarrow X\sslash_\chi G
\ee
establishes a one-to-one correspondence between 
the set of closed points of the quotient $X\sslash_\chi G$ and the set of
$S_\chi$-equivalence classes in $X(\chi)$ (equivalently, with the set of closed orbits).
The morphism 
\be\label{GIT-desing}
X\sslash_\chi G\longrightarrow X\sslash G=X\sslash_{\chi=0} G
\ee
is projective and, often, a resolution of singularities.

Let us turn our attention to quivers.

\begin{defin} A quiver $Q$ is an oriented graph. We denote by $I$ the vertex set and by $E$ the edge and we write $Q=(I, E)$. Given an oriented edge $e\in E$, one lets $s(e)$ and $t(e)$ denote the ``source''
and the ``target'' of $e$, respectively. If $e$ is a loop then $t(e)=h(e)$.
\end{defin}

Associated to a quiver $Q$  is the so called {\it Cartan matrix}. Set $|I|=s$, then the Cartan matrix is the $s\times s$ integral matrix
\be\label{cartan}
\mathbf C=(c_{ij})
\ee
defined by
\[
c_{ij}=\left\{ \begin{array}{cc}2-2\,\sharp\,(\text{edges joining}\,\, i \,\,\text{to itself}) & \text{ if } i=j \\
-\sharp\,(\text{edges joining}\,\, i \,\,\text{to}\,\,j) & \text{ if } i\neq j\\
\end{array} \right.
\]

We also set
$$
\mathbf D=-\mathbf C\,,\qquad d_{ij}=-c_{ij}
$$
We let $d: \ZZ^s \times \ZZ^s \to \ZZ$ be the quadratic form associated to $(d_{ij})$ and we set:
\be 
d( \mathbf n)= {}^t\mathbf  n D  \ \mathbf n.
\ee

Fix a {\it dimension vector}
$$
\mathbf n=(n_1,\dots, n_s)\in \ZZ_{\ge 0}^s
$$
and vector spaces $V_i$, $i=1,\dots,s$, with
$$
\dim V_i=n_i,
$$
and define the vector space of  {\it $\mathbf n$-dimensional representations of} $Q$
by setting
$$
\rep(Q, \mathbf n):=\underset{e\in E}\bigoplus \Hom (V_{s(e)}, V_{t(e)})
$$
The group
$$
G(\mathbf n):=\prod_{i=1}^s\GL(n_i)
$$
acts on $\rep(Q, \mathbf n)$ in a natural way via conjugation.
We denote by
$Q^{op}=(I, E^{op})$
the quiver with the same underlying graph as $Q$, where the orientation of every edge has been reversed. The trace  pairing gives an isomorphism
$$
\rep(Q^{op}, \mathbf n) \cong \rep(Q, \mathbf n)^\vee,
$$
Finally, one  defines a new quiver $\overline Q$ having the same set of vertices as $Q$ and having
$$
\overline E:= E \sqcup E^{op}
$$
as edge set. We then get an identification
$$
\rep(\ov Q, \mathbf n)=\rep(Q, \mathbf n)\oplus \rep( Q, \mathbf n)^\vee
$$
so that we get a natural symplectic form on $\rep(\ov Q, \mathbf n)$. Since the action of $G(\mathbf n)$ on $\rep(\ov Q, \mathbf n)$ respects the symplectic form there is a {\it moment map} which is given by
\be \label{moment map}
\aligned
\mu_{\mathbf n}: \rep(\ov Q, \mathbf n)& \longrightarrow \gl(\mathbf n)\cong \gl(\mathbf n)^\vee\\
(x,y^\vee) & \longmapsto\mu(x,y^\vee)=\sum_{e\in E}[x_e, y_e^\vee]
\endaligned
\ee
Here $\gl(\mathbf n)$ and $\gl(\mathbf n)^\vee$ are identified via the Killing form and $\mu_{\mathbf n}$ is $G(\mathbf n)$--equivariant with respect to the coadjoint action on $ \gl(\mathbf n)^\vee$. Since the center $\mathbb C^*$ of $G(\mathbf n)$ acts trivially, the moment map has values in the hyperplane $(Lie \, \CC^\times)^\perp \subset  \gl(\mathbf n)^\vee $.

Recall that rational characters 
\[
\chi: G(\mathbf n)\to \CC^\times
\]
are in a one-to-one correspondence with vectors
\[
\theta=(\theta_1,\dots,\theta_s)\in \ZZ^s
\]
via the formula 
$$
\chi_{\theta}(g)=\prod\det(g_i)^{\theta_i}
$$
where $g=(g_1,\dots,g_s)\in G(\mathbf n)$. To simplify notation we will freely substitute the symbol $\chi_\theta$ with $\theta$.

In general, the space $\rep(Q, \mathbf n)\sslash_{\chi_{\theta}} G(\mathbf n)$ is very singular, and
the philosophy behind the moment map is that
the quotient 
$$
\mathfrak M_{\theta}(\mathbf n):=\mu_{\mathbf n}^{-1}(0)\sslash_{\chi_{\theta}} G(\mathbf n)
$$
is the natural substitute for the non-existing tangent bundle to $\rep(Q, \mathbf n)\sslash_{\chi_{\theta}} G(\mathbf n)$. 
This assertion is justified by the following process, called Marsden--Weinstein or symplectic reduction (\cite{Marsden-Weinstein}, \cite{Ginzburg}). Let
$$
\pi: \mu_{\mathbf n}^{-1}(0)(\chi_{\theta}) \longrightarrow \mathfrak M_{\theta}(\mathbf n)
$$
be the quotient morphism.
At a smooth point $x \in \mu^{-1}(0)$, the tangent space to the  orbit $G(\mathbf n) \cdot x$ is the orthogonal complement (with respect to natural symplectic structure on $\rep(\ov Q, \mathbf n)$) to the tangent space at $x$ to $\mu^{-1}(0)$. Hence, the normal space $N_x$ carries a natural sympletic structure. If in addition $x$ is ${\chi_{\theta}}$-stable, then its $S_{\chi_{\theta}}$-equivalence class coincides with its orbit,  the point $\pi(x)$ is a smooth point of $\mathfrak M_{\theta}(\mathbf n)$ and $N_x$ can therefore be identified with the tangent space to $\mathfrak M_{\theta}(\mathbf n)$ at  $\pi(x)$.
We denote by $\mathfrak M^s_{\theta}(\mathbf n)$ the locus parametrizing orbits of stable points in $\mathfrak M_{\theta}(\mathbf n)$, so that
$$
\mathfrak M^s_{\theta}(\mathbf n)\subset\mathfrak M_{\theta}(\mathbf n)_{\text{smooth}}
$$
has a natural holomorphic symplectic form defined by Marsden-Weinstein reduction.
We have
$$
\dim \mathfrak M_{\theta}(\mathbf n)=d(\mathbf n)+2.
$$
Following Crawley-Boevey \cite{Crawley-Boevey} define $p(\mathbf n)$ by
\[
d(\mathbf n)+2=2p(\mathbf n)
\]
so that
\[
\dim \mathfrak M_{\theta}(\mathbf n)=2p(\mathbf n)
\]

The notion of $\chi_{\theta}$-semistability can also be described in terms of a slope function, which was first introduced by King in \cite{King}. 

Given a quiver $Q$, consider in $\ZZ^s$ the orthogonal complement of the dimension vector $\mathbf n$:
$$
\mathbf n^\perp\subset \ZZ^s,
$$
and consider $\theta\in \mathbf n^\perp\otimes \QQ$.
Let $V= \oplus V_i$ be an $ \mathbf n$-dimensional representation of $Q$ (or of $\ov Q$). For any sub-representation 
$$
W=\oplus_{i\in I} W_i,\,\qquad W_i\subset V_i
$$
we define the $\theta$-slope of $W$ by setting 
by setting 
$$
{\slope}_\theta(W)=\ff{\theta\cdot\dim W}{\sum \dim W_i }=\ff{\sum_{i=1}^s\theta_i\dim W_i}{\sum \dim W_i }
$$
so that, in particular, $\slope_\theta(V)=0$. Accordingly, a non-zero representation $V$ is said to be {\it $\theta$-semistable} if, for every sub-representation $W$ of $V$, we have $\slope_\theta(W)\leq 0$ and is said to be {\it $\theta$-stable}, if  the strict equality holds for every non-zero, proper sub-representation.

\begin{rem} \label{theta zero}
Consider $\theta=\mathbf 0=(0, \cdots, 0)$. Then
\[
{\slope}_{\mathbf 0}(W)=0
\]
for any $W=\oplus W_i$, so that any representation $V$ is $\mathbf 0$-semistable. Moreover, the $\mathbf 0$-stable representations are precisely the \emph{simple} ones, i.e., those that have no non trivial sub-representations. Lastly, a simple representation is stable with respect to any $\theta$.
\end{rem}

As usual in this context, given a $\theta$-semistable one can consider a Jordan-H\"older filtration and then the associated graded $\gr_\theta(V)$. For example, if $\theta=\mathbf 0$, then the Jordan-H\"older filtration is just a composition series for $V$, and $\gr_{\mathbf 0}(V)$ is a direct sum of simple representations, the so-called ``semi-simplification'' of $V$.
Two $\mathbf n$-dimensional $\theta$-semistable representations $V$ and $V'$ are called $S_\theta$-equivalent, if
\[
\gr_\theta(V)=\gr_\theta(V').
\]

\begin{theorem}[King \cite{King}] Let $ Q$ be a quiver, and $V$ a representation of $Q$, or of $\ov Q$, with dimension vector $\mathbf n$. Let $\theta\in \mathbf n^\perp$. Then 
\begin{itemize}
\item[{\rm i)}] A representation $V$ is $\theta$-semistable (resp. $\theta$-stable) if and only if the point $[V] \in \Rep(Q, \mathbf n)$ (or in $\Rep(\ov Q, \mathbf n)$) is $\chi_\theta$-semistable (resp. $\chi_\theta$-stable);\\
\item[{\rm ii)}] Two $\mathbf n$-dimensional $\theta$-semistable representations $V$ and $V'$ are $S_\theta$-equivalent if and only if the corresponding points in $\Rep(Q, \mathbf n)$ (or in $\Rep(\ov Q, \mathbf n)$) are $S_{\chi_\theta}$-equivalent.
\end{itemize}
\end{theorem}

\begin{rem} \emph{
For the definition of slope, we follow \cite{Nakajima-branching} and \cite{Ginzburg}, even though it differs from the one considered in \cite{King} by a sign. However, taking the dual of a representation preserves stability so that $  \mathfrak M_{\theta}(\mathbf n)$ is canonically isomorphic to $ \mathfrak M_{-\theta}(\mathbf n)$. Hence, from our point of view, this change of sign is irrelevant.}
\end{rem}

\begin{rem}\label{center acts trivially}\emph{
The reason to consider $\theta \in \mathbf n^\perp$ is that in this way the character is trivial when restricted to the center $\CC^ \times \subset G(\mathbf n)$, which acts trivially on the $\mathbf n$--dimensional quiver representations of $\ov Q$ (cf. the remark after Proposition \ref{Hilb Mum} and the {\it Warning} on page 517 of \cite{King}). From the point of view of $\mu_\theta$ stability, it is not strictly necessary to assume that $\sum \theta_i n_i=0$. However, as observed by Rudakov (\cite{Rudakov} Proposition 3.4, cf. also Remark 2.3.3 of \cite{Ginzburg}), it is always possible to reduce to this case since stability with respect to a given $\theta$ is equivalent to stability with respect to $\theta-c(1, \dots, 1)$, for any constant $c$. 
Another way of solving this issue would be to consider instead the action of the group $G(\mathbf n) \cap SL(\oplus V_i)$. Since the first seems to be the convention adopted widely in this context, we stick to it. From now on we will assume that $\theta \cdot \mathbf n=0$.
}
\end{rem}

Given an element $\alpha=(\alpha_1,\dots,\alpha_s) \in \ZZ^s$ the {\it support } of $\alpha$, denoted by
$\supp(\alpha)$,  is the subgraph of $Q$
consisting of those vertices $i$ for which $\alpha_i\neq0$  and all the edges joining these vertices.

Kac  has generalized the concept of positive roots to arbitrary quivers (not only of Dynkin type):
$$
R_+:=\{\alpha \in \ZZ^s_+\,\,|\,\, d(\alpha) \ge 2 \,\,\, \text{and } \,\supp(\alpha)\,\,\,\text{is connected}\}.
$$
and has shown in \cite{Kac} and \cite{Kac-root} that there exist an indecomposable representation of a given dimension vector precisely if the dimension vector is a positive root. We can now state the first  of the two theorems by Crawley-Boevey that we will need.

\begin{theorem}{\rm(Crawley-Boevey \cite{Crawley-Boevey}, Theorem 1.2)} \label{CBT1}Let $Q$ be a quiver with $s$ vertices and let $\mathbf n\in \ZZ^s_+$ be a dimension vector. Then there exists a simple representation in $\mu_{\mathbf n}^{-1}(0)$ if and only if $\mathbf n$ is a positive root and, for any decomposition
$$
\mathbf n=\beta^{(1)}+\cdots \beta^{(r)}\,,\quad r\geq 2\,,\quad \beta^{(i)} \in R_+ ,\,\,\ \text{ for }\,\, i=1,\dots r
$$
the inequality
\be \label{inequality dimensions}
p(\mathbf n)>\sum_{i=1}^r p(\beta^{(i)})
\ee
holds. In this case $\mu_{\mathbf n}^{-1}(0)$ is a  reduced and  irreducible complete intersection of dimension 
$$
2p(\mathbf n)+{}^t\mathbf n\cdot\mathbf n -1=d(\mathbf n)+2+{}^t\mathbf n\cdot\mathbf n -1.
$$
\end{theorem}

Let $V$ be a semisimple (i.e., $0$-semistable) representation, and consider its simple components
$$
\gr_{\mathbf 0}(V)=V_1^{k_1}\oplus\cdots\oplus V_r^{k_r}\,,\quad \dim V_i=\beta^{(i)}\,,\quad i=1,\dots,r.
$$
One then says that $V$ has type $\tau=(k_1, \beta^{(1)};\dots; k_r, \beta^{(r)} )$. Notice that the representation types $\tau$ are in one-to-one correspondence with decompositions
\[
\mathbf n=k_1\beta^{(1)}+\cdots +k_r \beta^{(r)},
\]
with $k_i >0$ and $\beta^{(i)}$ a positive root.
The second theorem of Crawley-Boevey is the following.

\begin{theorem}{\rm(Crawley-Boevey \cite{Crawley-Boevey}, Theorem 1.3)}\label{CBT2} Let $Q$ be a quiver with $s$ vertices. Let $\mathbf n\in \ZZ^s_+$ be a dimension vector. Suppose $\mathbf n=k_1 \beta^{(1)}+\cdots+k_r\beta^{(r)}$.
Then the set  $\Sigma_\tau$ of semisimple representations of type $\tau=(k_1, \beta^{(1)};\dots; k_r, \beta^{(r)} )$
is a locally closed  subset of  $\mathfrak M_{0}(\mathbf n)=\mu_{\mathbf n}^{-1}(0)\sslash G(\mathbf n)$ of dimension $2\sum_{i=1}^r p(\beta^{(i)})$.
\end{theorem}

Set
$$
\W_\mathbf n=\mathbf n^\perp\otimes\QQ\subset\QQ^s
$$
As we already observed,  points of  $\W_\mathbf n$ may be thought of as stability parameters for quiver representations. 
Nakajima \cite{Nakajima-branching}, introduced a wall and chamber structure in $\W_\mathbf n$, which we now describe. We set
$$
R_+(\mathbf n)=\{\alpha\in\ZZ^s\,\,|\,\,  \alpha\,\,\text{a positive root,}\,\,\,\alpha_i\leq n_i\}\setminus\{0\,,\mathbf n\}
$$
(here, to avoid  ``redundant'' walls we slightly depart from Nakajima's definition by adding the condition on the connected support). By virtue of Theorem \ref{CBT1}, if $\alpha$ belongs to $R_+(\mathbf n)$ then
$d(\alpha)+2\geq 0$.
For every $\alpha\in R_+(\mathbf n)$ we define {\it the wall associated to $\alpha$} by setting
\be \label{W alpha}
\W_\alpha=\{\theta\in \W_\mathbf n\,\,|\,\,\theta\cdot\alpha=0\}
\ee
The idea is that there exist a strictly $\theta$-semistable $V$ with an $\alpha$-dimensional sub-representation $V' \subset V$ with $\slope_\alpha(V')=\slope_\alpha(V)$, precisely when $\theta$ lies in $\W_\alpha$. Notice that
\[
\text{if}\,\,\,\, \alpha+\beta=\mathbf n\,,\,\,\,\,\text{then}\,\,\,\,\W_\alpha=\W_\beta
\] 
By definition, the  {\it chambers of  $\W_\mathbf n$} are the connected components of the complement of the walls.
A point of  $\W_\mathbf n$ is said to be {\it $\mathbf n$-generic} if it lies in a chamber. In Nakajima's language, a codimension $i \ge 1$ {\it face} of $\W_\mathbf n$ is a connected component of the complement of the intersection of $(i+1)$ walls in an intersection of $i$ walls. One of Nakajima's result is the following, which is the quiver counterpart of Theorem \ref{v walls} and Proposition \ref{walls and strata}.

\begin{prop}[Nakajima \cite{Nakajima-branching}, Lemma 2.12 ]\label{Nak-lem}

\item[{\rm (1)}] If $\theta$ is in a chamber then $\theta$-semistability implies $\theta$-stability so that
$\mathfrak M^s_{\theta}(\mathbf n)=\mathfrak M_{\theta}(\mathbf n)$.

\item[{\rm (2)}] If two stability parameters $\theta$ and $\theta'$ are contained in the same face, then
$\theta$-semistability (resp. $\theta$-stability) is equivalent to $\theta'$-semistability (resp. $\theta'$-stability).

\item[{\rm (3)}] Let $F$ and $F'$ be faces such that $F'\subset\ov F$. Suppose that $\theta\in F$ and $\theta'\in F'$. Then:

{\rm (i)}  a $\theta$-semistable representation is also $\theta'$-semistable,

{\rm (ii)} a $\theta'$-stable representation is also $\theta$-stable.

\end{prop}

In particular, since all the faces contain   $\mathbf 0$ in their closure, for any $\theta \in \mathbf n^\perp$, there is a natural projective morphism
\[
\xi: \mathfrak M_{\theta}(\mathbf{ n}) \longrightarrow \mathfrak M_0(\mathbf{ n}),
\]
which is an isomorphism on the locus of simple representations. Recall, also, that $S_{\mathbf 0}$-equivalence classes of representations are in one-to-one correspondence with isomorphism classes of direct sum of simple representations and therefore one can interpret the morphism $\xi$ as the ``semisimplification'' map, which to a representation $V$ assigns the isomorphism class of $\gr_{\mathbf 0}(V)$:
\[
\xi: V \longmapsto \gr_{\mathbf 0}(V).
\]

Finally, observe that if $\ov Q$ and $\mathbf n$ are such that (\ref{inequality dimensions}) holds for any decomposition $\mathbf n=\beta^{(1)}+\cdots +\beta^{(r)}$, then the assumptions of Theorem \ref{CBT1} are satisfied, hence the simple locus is non-empty and $\xi$ is birational. As a consequence, if $\theta$ is $\mathbf n$-generic so that $\mathfrak M^s_{\theta}(\mathbf n)\subset\mathfrak M_{\theta}(\mathbf n)$, then $\xi $ is a symplectic resolution.

\begin{rem}[\cite{Ginzburg}, Remark 2.3.10] \label{dual}
There is a canonical isomorphism $\mathfrak M_{\theta} \cong \mathfrak M_{-\theta}$, given by taking the dual represention.
\end{rem}

\section{Statement of the main theorem}\label{Statement of the main theorem}

Before stating the main theorem we show how to associate a quiver to a polystable $F$ on a K3 surface $S$. The connection with quiver varieties is already present in Kaledin, Lehn and Sorger who pointed out in \cite{Kaledin-Lehn-Sorger06} the strong similarity between singular moduli spaces and Nakajima quiver varieties (\S 2.7 of loc. cit).

\begin{prop} \label{quiver} Let $H_0$ be a polarization on S, let $V_1, \dots, V_s$ be vector spaces of dimension $n_1, \dots, n_s$ and let $F_1, \dots, F_s$ be pairwise distinct $H_0$-stable sheaves such that the sheaf
\[
F=\oplus_{i=1}^s F_i \otimes V_i,
\]
is $H_0$-polystable. Set $\mathbf{n}:=(n_1, \dots, n_s)$ and $G(\mathbf{n})=\prod GL(n_i)$, so that
\[
G(\mathbf{ n}) \cong \Aut(F).
\] 
There exist a quiver $Q=Q(F)$ and 
$G(\mathbf{n})$-equivariant isomorphisms
\[
\Rep(\overline Q, \mathbf{n}) \cong \Ext^1(F,F), \quad  \mathfrak{gl}(\mathbf{n})^\vee \cong \Ext^2(F,F),
\]
such that, via these isomorphisms, the quadratic part (\ref{k2 cup}) of the Kuranishi map for $F$
\[
k_2: \Ext^1(F, F) \to \Ext^2(F,F),
\]
corresponds to the moment map (\ref{moment map}).

\end{prop}
\begin{proof} For brevity, we use the notation
$$
\dext^i(A,B)=\dim\Ext^i(A,B)
$$
We first define 
the quiver $Q(F)$: the vertex set of $Q$ is the set $I=\{1,\dots,s\}$ of distinct stable factors of $F$; the number of edges between the $i$-th and the $j$-th vertex is equal to
\[
\left\{ \begin{array}{ll} \dext^1(F_i, F_i) /2 & \text{ if } i=j \\
\dext^1(F_i, F_j) & \text{ if } i\neq j
\end{array} \right.
\]
Since we will be passing to the quiver $\ov Q$, we can choose for each of these edges and loops an arbitrary orientation. The Cartan matrix (\ref{cartan}) is then defined by
$$
c_{ij}=\left\{ \begin{array}{cc}2-\dext^1(F_i,F_i)& \text{ if } i=j \\
-\dext^1(F_i,F_j)& \text{ if } i\neq j\\
\end{array} \right.
$$
we now pass to the double $\ov Q$ of $Q$ and we have
$$
\begin{aligned}
\Rep(\overline Q, \mathbf n)& =\bigoplus_{i=1}^s \End(V_i)^{\oplus \dext^1(F_i,F_i)}  \bigoplus_{i<j} \Big(  \End(V_i, V_j)^{\oplus \dext^1(F_i,F_j)} \oplus \End(V_j, V_i)^{\oplus \dext^1(F_j,F_i)} \Big) \\
&\cong \bigoplus_{i=1}^s \End(V_i)\otimes \Ext^1(F_i, F_i)  \bigoplus_{i<j} \Big( \End(V_i, V_j)\otimes \Ext^1(F_i, F_j)  \oplus \End(V_j, V_i)\otimes \Ext^1(F_j, F_i)  \Big)\\
&=\Ext^1(F, F).
\end{aligned}
$$
In a similar way
$$
\mathfrak{gl}(\mathbf n)=\overset{s}{\underset{i=1}\oplus}\Hom(V_i, V_i)=\Hom(F,F)=\Ext^2(F,F)^\vee
$$

The fact that  via these isomorphisms, the quadratic part of the Kuranishi map is a moment map (\ref{moment map}) is explained in section 3.4 of
 \cite{Kaledin-Lehn-Sorger06} and was already present in \cite{OGrady99}.
 
\end{proof}

A few remarks are in order. First of all, it should be pointed out that the construction of the quiver $Q(F)$ associated to the polystable sheaf $F$ can been done in full generality, without any restrictions on $F$.

In the case where $F$ is pure of dimension one, there is the following interpretation of $Q(F)$.
\be\label{divisor}
D_i=\supp F_i\,,\quad i=1,\dots,s\,,\qquad D=n_1D_1+\cdots+n_sD_s
\ee
(here $\supp( \cdot)$ denotes the Fitting support) so that
\be\label{divisor2}
[D]=c_1(F)
\ee
We also set 
\be\label{divisor3}
\aligned
\chi_i&=\chi(F_i)\,,\quad i=1,\dots,s\,,\qquad \chi=\chi(F)\\
v_i&=(0, [D_i], \chi_i)\,,\quad i=1,\dots,s\,,\qquad v=(0,[D], \chi)=\sum_{i=1}^sn_iv_i
\endaligned
\ee
Notice that
\[
\dext^1(F_i, F_i)=\dim M_{H_0}(v_i)=\frac{D_i^2}{2}+2=g(D_i)\,,\quad i=1,\dots,s\,,
\]
and that if $i \neq j$, then
\[
\dext^1(F_i, F_j)=D_i\cdot D_j.
\]

It follows that we can think of the quiver $Q$ as the ``dual graph'' of $\ov D:=D_1+\cdots+ D_s$ (in the sense that it has a vertex for every curve $D_i$ and for every $i \neq j$ it has $D_i \cdot D_j$ edged connecting $i$ and $j$), with $g(D_i)$ loops attached to the $i$th vertex.
It is also worth mentioning that if the K3 surface is general enough (e.g. it contains no rational curve), then we can deform each sheaf $F_i$ to a sheaf $F_i'$ with smooth support (this clearly does not alter the structure of the singularity) so that, up to the addition of the vertex loops, $Q$ is in fact the dual graph of a curve.

\begin{rem} \label{symmetries}\emph{
Suppose that there are two indices, say $i=1, 2$, for which the two curves $D_1$ and $D_2$ belong to the same linear system. Then, for every $j=1, \dots, s$, we have $D_1 \cdot D_j= D_2 \cdot D_j$, so that the quiver $\ov Q$ admits a symmetry which swaps the first and the second vertices. More generally, partitioning the index set $I=\{1, \dots, s\}$ according to the cohomology class of the curve of each vertex, we can define the subgroup $\mathcal G \subset \Aut(\ov Q)$ of the symmetries of $\ov Q$ preserving the curve class of every vertex.
}
\end{rem}

The last proposition allows us to start comparing the moduli space side of the picture with the quiver side. First some notation.

For any $\mathbf \beta=(\beta_1, \dots, \beta_s) \in \ZZ_{\ge 0}^s$, define
\[
v(\beta):= \sum_{i=1}^s \beta_i v_i\in H^*(S, \ZZ)
\]
so that $v(\beta)=v( \oplus F_i^{\beta_i})$. Notice that $v(\mathbf n)=v$ and that as soon as $\beta \neq 0$, $v(\beta)$ is a positive Mukai vector.

\begin{prop} \label{comparing 1} Let $F$ be the $H_0$ polystable as above, let $Q=Q(F)$ and $\mathbf n$ be as in Proposition \ref{quiver} and let $R_+$ be the set of positive roots for $Q$.
\begin{itemize}
\item [{\rm (1)}]For any $\mathbf \beta \in \ZZ_{\ge 0}^s$ we have $v(\beta)^2=d(\beta)$.
In particular, the moduli space $M_{H_0}(v(\beta))$ is non-empty and
\[
\dim M_{H_0}( v(\beta))=d(\beta)+2=2p(\mathbf \beta).
\]
\item[{\rm (2)}] The moduli space $M_{H_0}(\beta)$ contains a stable sheaf if and only if $\beta$ lies in $R_+$.
\item [{\rm (3)}] For $\mathbf n \in R_+$, decompositions 
$$
\mathbf n= \sum_{j=1}^r k_j \beta^{(j)}
$$
 with $\beta^{(j)}\in R_+(\mathbf n)$ and $k_j >0$, $j=1,\dots,r$,  are in one-to-one correspondence with the strata of the singular locus of $M_{H_0}(v)$ containing the polystable sheaf $F$ in their closure. In particular, the equations of $v$-walls that are relevant to $F$ (recall Definition \ref{relevant})
 are of the form
\be \label{equations walls concerning F}
\chi \, \left(\sum_{i=1}^s \beta_i D_i \cdot x\right)=\chi_\beta (D \cdot x ),\qquad x\in \Amp(S)\otimes_\ZZ\QQ
\ee
for some uniquely determined $\chi_\beta \in \ZZ$.
\end{itemize}
\end{prop}
\begin{proof}
The first two statements are immediate consequence of the definitions and of the description of the singular locus of $M_{H_0}(v)$ given in Proposition \ref{walls and strata}. Consider a decomposition $\mathbf n= \sum_{j=1}^r k_j \beta^{(j)}$, with $\beta^{(j)}\in R_+(\beta)$,  $j=1,\dots,r$. By the first two statements we know that for each $\beta^{(j)}$, the stable locus of $M_{H_0}(v(\beta^{(j)}))$ is non-empty, so we can associate to the decomposition above the strata parametrizing polystable sheaves of the form $\overset{r}{\underset{j=1}{\oplus}} F(\beta^{(j)})^{k_j}$, where for each $j$, the sheaf $F(\beta^{(j)})$ is a $H_0$-stable sheaf in $M_{H_0}(v(\beta^{(j)}))$. To see that these are the strata containing $F=\overset{s}{\underset{i=1}{\oplus}} F_i^{n_i}$ in their closure, we only need to notice that within each $M_{H_0}(v(\beta^{(j)}))$ we can deform the stable sheaves $F(\beta^{(j)})$ to the polystable sheaf $\overset{s}{\underset{i=1}{\oplus}} F_i^{\beta_i^{(j)}}$ whose support is
\be\label{decompD}
 \Delta_j=\beta^{(j)}_1D_1+\cdots+\beta^{(j)}_sD_s\,,\quad j=1,\dots, r
\ee
 so that
\be\label{decomp-Del}
 D=n_1D_1+\cdots+n_sD_s=k_1\Delta_1+\cdots+k_r\Delta_r
\ee
In this way we assign to each decomposition a stratum containing $[F]$ in its closure.
The description of the converse assignement 
is left to the reader.
\end{proof}

Consider the setting and the notation of Proposition \ref{quiver}. As our aim is to study the singularity of $M_{H_0}(v)$ at $[F]$ and its symplectic resolutions induced by the polarizations which are adjacent to $H_0$, we only need to focus on the $v$-walls in $\Amp(S)$ that contain $H_0$. We now show that such $v$-walls correspond to the walls in $\mathbf n^ \perp$ as described above.

We first need some notation.
Set
\[
d_i:=H_0 \cdot D_i, \quad d:=\sum n_i d_i = H_0 \cdot D, \quad \mathbf d:=(d_1, \dots, d_s),
\]
and for any ample $H$
\[
a_i:=H \cdot D_i, \quad  h=\sum n_i a_i = H \cdot D,\quad \mathbf a:=(a_1, \dots, a_s)
\]
By Proposition \ref{comparing 1}, since the only things that matters for our purpose are the intersection numbers of $H$ with the curves of the form $\sum \beta_i D_i$, we can project the ample cone of $S$ onto the cone
\[
\mc A=\{\mathbf a=(a_1, \dots, a_s) \in \QQ^s,\,\, a_i \ge 0 \},
\]
and consider instead the $v$-walls in $\mc A$. Under this projection, the class of $H_0$ is sent to the point $\mathbf d \in \mc A$.
Since stability with respect to a given polarization only depends on the positive ray determined by the polarization itself, we can consider instead of $\mc A$ the transverse slice
\[
\mc S=\{ \mathbf a \in \mc A \, | \sum_i a_i n_i=d \}.
\]
In this space the equations (\ref{equations walls concerning F}) of the $v$-walls that pass through $H_0$ and that make $F$ strictly semistable become
\be \label{walls in S}
\chi \, \sum a_i  \beta_i - d \, \chi_\beta=0,
\ee
where
\[
\chi_\beta=  \ff{\chi}{d} \sum d_i \beta_i.
\]

\begin{lem} \label{affine morphism}
The affine morphism
\[
\begin{aligned}
\Xi: \mc S &\longrightarrow {\W}_{\mathbf n}=\mathbf n^\perp \otimes \QQ, \\
(a_1, \dots, a_s) & \longmapsto (a_1-d_1,  \dots, a_s-d_s) 
\end{aligned}
\]
sends $\mathbf d$ to the origin and maps every $v$-wall that is relevant to $[F]$ to a wall in ${\W}_{\mathbf n}$. More specifically, it maps the wall $\{\chi \, \sum a_i  \beta_i- d \, \chi_\beta=0 \}$ to the wall $\W _\beta$, where $\beta=(\beta_1, \dots, \beta_s)$ (notation as in (\ref{W alpha})).
\end{lem}
\begin{proof} In accordance with the notation of Section \ref{quivers}, we let $(\theta_1, \dots, \theta_s)$ be the coordinates on $\W_{\mathbf n}$.
Subsituting $\theta_i=a_i -d_i$ in (\ref{walls in S}) we find
\[
\chi \, \sum \beta_i \theta_i  +\chi \sum \beta_i d_i- d \, \chi_\beta=0,
\]
and since $d\chi_\beta=\chi \sum d_i \beta_i$ we get
\[
\chi \, \sum \beta_i \theta_i =0
\]
which is the equation for $\W_\beta$.
\end{proof}

Notice that if the group $\mathcal G$ defined in Remark \ref{symmetries} is non-trivial, the image of $\Xi$ is not the whole of $\W_\mathbf n$, but is the $\G$-invariant subspace $\W^{\G}_\mathbf n \subset \W_\mathbf n$ and, similarly, the walls that come from $\mc S$ are the  walls $\W_\beta$ for which $\beta$ is $\G$-invariant.

Finally, we get to the statement of the main theorem, whose proof will cover Section \ref{proof of main theorem}.

\begin{theorem} \label{main}
Let $H_0$ be a polarization on $S$ and let $F_1, \dots, F_s$ be pairwise distinct $H_0$-stable sheaves. Let $V_1, \dots, V_s$ be vector spaces of dimension $n_1, \dots, n_s$ respectively,  let
\be\label{basic F}
F=\oplus_{i=1}^s F_i \otimes V_i,
\ee
be the corresponding $H_0$-polystable sheaf and let $v$ be its Mukai vector. Also set:
$$
G=\Aut(F)=\overset s{\underset{l=1}\prod }\GL(V_i)
$$

\begin{enumerate}
\item[ {\rm (i)}] Suppose that $F$ is pure of dimension one (or satisfies the formality property of Definition \ref{formality}). Then there is a local (analytic) isomorphism
\[
\psi: (\mathfrak{M}_0,0) \cong (M_{H_0}(v), [F] )
\]
\item[{\rm (ii)}] Suppose that $F$ is pure of dimension one. Then for every chamber $\mathfrak C \subset \Amp(S)$ containing $H_0$ in its closure, we can find a chamber $\mathfrak D \subset \mathbf{n}^\perp$ such that for every $H \in \mathfrak C $ and every $\theta \in  \mathfrak D$ the symplectic resolutions
\[
\xi: \mathfrak M_{\theta}(\mathbf{ n}) \to \mathfrak M_0(\mathbf{ n}), \quad \text{ and } \quad h: M_H(v) \to M_{H_0}(v),
\]
correspond to each other via $\psi$. This means that, letting $\ov \U  \subset M_{H_0}(v)$ and $\ov \V \subset \mathfrak M_0(\mathbf{ n})$ be two open neighborhoods of $[F]$ and $0$, respectively, that are isomorphic via $\psi$, there is a commutative diagram
\[
\xymatrix{
 M_H(v) \times_{h} \ov \U \ar[d] \ar[r] & \mathfrak M_{\theta}(\mathbf{ n}) \times_{\xi}\ov \V \ar[d]\\
 \ov\U \ar[r]_\psi & \ov \V
}
\]
\item[{\rm (iii)}] The assigment of a chamber in $\mathbf{n}^\perp\otimes \mathbb Q$ for every chamber in $\Amp(S)$ which is adjacent to $H_0$ is induced by the morphism of Lemma \ref{affine morphism}. In other words, if $H$ is such that $H \cdot D= H_0 \cdot D$ the morphism is given by the formula
$$
H \longmapsto \chi_H((g_1,\dots, g_s))=\prod_{i=1}^s\det(g_i)^{(D_i\cdot H-D_i \cdot H_0)}
$$ 
where $D_i=c_1( F_i)$, for $i=1,\dots, s$.
\end{enumerate}
\end{theorem}

\begin{rem} \emph{
Whether or not $F$ is a pure dimension one sheaf, statements $(ii)$ and $(iii)$ of the theorem holds true whenever the morphism $h: M_H(v) \to M_{H_0}(v)$ is regular over $F$.
}
\end{rem}

\section{Proof of the main Theorem}\label{proof of main theorem}

We consider as in (\ref{Q slice}) an \'etale slice 
\be\label {Q slice2}
Z \subset \Quot_{H_0}
\ee
passing through a point $q_0$ corresponding to the $H_0$ polystable sheaf 
\be\label{basic F 2}
F=\oplus_{i=1}^s F_i \otimes V_i,
\ee
and we let $\F$ be the restriction to $Z \times S$ of the universal family over $\Quot_{H_0} \times S$.

Let us start with the proof of part $(i)$, which is straightforward.  By Proposition \ref{quiver} there is a quiver $Q$ such that
\be \label{mu e k2}
\mu_{\mathbf n}^{-1}(0) \cong k_2^{-1}(0)
\ee
$G(\mathbf n)$-equivariantly. For simplicity, we set
$$
G:=G(\mathbf{n}).
$$
Recall that $G \cong \Aut(F)$.

By Theorem \ref{stabil-LM} the Lazersfeld--Mukai bundle $M_F$ is polystable and hence by Zhang's result (Theorem \ref{Zhang}), it satisfies the formality property. Using Proposition  \ref{Q-is-a-cone} applied to $M_F$ and Proposition \ref{defm and deff}, it follows that there is a local $G(\mathbf{n})$-equivariant isomorphism $Z \cong k_2^{-1}(0)$, which induces, locally around $0$ and $[F]$, respectively, an isomorphism between $Z\sslash G $ and $\mathfrak{M}_0$, .
Since the morphism
\[
\epsilon: Z\sslash G \to M_{H_0}(v)
\]
of (\ref{etale morphism}) is \'etale
we may conclude that $\mathfrak{M}_0$ and  $M_{H_0}(v)$ are isomorphic, locally around $0$ and $[F]$, respectively.

The proof of part (ii) will be divided in various steps. Consider the resolution
\[
h: M_H(v) \to M_{H_0}(v).
\]
Our aim is to show that locally on $M_{H_0}(v)$  the resolution $h$  can be expressed, via quiver varieties, 
in terms of variations of GIT quotients as in (\ref{GIT-desing}). We will do this in two steps, the first consists in using the open subset of $Z$ parametrizing the $H$-semistable sheaves, and the second will be to compare this open subset with an appropriate open subset of $\mu_{\mathbf{n}}^{-1}(0)$.

\subsection{First step}\label{First step}

Let $H \in \Amp(S)$ be a polarization that is adjacent to $H_0$, and let
\be\label{semist-locus}
Z^H =\{q\in Z \,\,|\,\,F_q\,\,\text{is $H$-semistable}\,\,\}
\ee
be the locus parametrizing $H$-semistable points in  $Z$. The restriction to $Z^H$ of the family $\F$ in (\ref{family on slice}) defines a classifying morphism $Z^H \to M_H$. Since this morphism is $G$-invariant we get a commutative diagram
\be \label{diagram}
\xymatrix{
M_H(v)\ar[d]_h &  Z^H \sslash G\ar[d]^\rho \ar[l]_{\,\,\,\,\,\,\,\,\,\eta\quad\,\,} \\
M_{H_0}(v)& Z \sslash G \ar[l]_\epsilon
}
\ee

\begin{prop} \label{MH e ZH}
The diagram above is cartesian.
\end{prop}

Before proving the proposition we need a technical lemma, which uses the fact that
the image of the natural morphism (recall the notation (\ref{etale morphism}))
\[
\GL(\VV) \times Z \to \Quot_{H_0}^{ss}
\]
is a saturated open subset. Following \cite{Drezet-Luna}, the precise statement we will use is

\begin{lem}[\cite{Drezet-Luna}, page 2] \label{proprieta slice} 
Let $\Gamma$ be a reductive algebraic group acting on an affine variety $Y$. Let $y \in Y$ be a point whose orbit $\Gamma y$ is closed, let $\Gamma_y$ be the stabilizer of $y$ in $\Gamma$, and let $Z \subset Y$ be an \'etale slice for $y$ in $Y$. Then for every point $y' \in Y$ that is $S_{\Gamma}$-equivalent to a point $z\in Z$, the slice $Z$ intersects the orbit $\Gamma y'$. In other words, the slice $Z$ intersects all the $\Gamma$-orbits that are $S_{\Gamma}$-equivalent to the $\Gamma$-orbits of its points. Moreover, given $z \in Z$, the natural morphism
\be \label{drezet 2}
\sigma: \Gamma \times \pi^{-1}_Z(\pi_Z(z)) \to \pi^{-1}_Y(\pi_Y(z))
\ee
is surjective and $\Gamma_y$--invariant and
\be \label{drezet 3}
\Gamma \times \pi^{-1}_Z(\pi_Z(z)) \sslash \Gamma_y \to \pi^{-1}_Y(\pi_Y(z))
\ee
is an isomorphism.
\end{lem}

\begin{lem} \label{bijective}
Let $H$ be a polarization adjacent to $H_0$,  let $Z^H \subset Z$ be the open subset parametrizing $H$-semistable points. Then, referring to diagram (\ref{diagram}), for every point $z \in Z \sslash G$,  the morphism $\eta$ induces a bijection between $\rho^{-1}(z)$ and $h^{-1}(\epsilon(z))$. 
\end{lem}
\begin{proof} Recall that the points of $M_H(v)$ correspond to $S_H$-equivalence classes of $H$-semistable sheaves.
Lemma \ref{proprieta slice} tells us  that for every $H$-semistable sheaf $F'$ whose isomorphism class lies in $h^{-1}(\epsilon(z))$, there exists a  point $b \in Z^H$ such that $\F_b \cong F'$. This proves that $\eta:  \rho^{-1}(z) \to h^{-1}(\epsilon(z))$ is surjective.  As for injectivity, we argue as follows.  Let
\[
\pi_Z:  Z\to Z\sslash G
\]
be the quotient map and let $\eta' :Z^H\to M_H(v)$ 
be a map inducing $\eta$.
Let $x$ and $y$ be two points in $Z^H\cap \pi_Z^{-1}(z)$, such that $\eta'(x)=\eta'(y)$.  This means that the two sheaves $\F_x$ and $\F_y$ are $S_H$-equivalent. 
We must prove that $x$ and $y$ are $S$-equivalent in $Z^H$, i.e. that the closure of their $G$--orbits intersect in $Z^H$.
In the $S_H$-equivalence class of  $\F_x$ there is a unique up to isomorphism $H$-polystable sheaf,  which we will denote by $F'$. This sheaf  is  $S_{H_0}$-equivalent to $\F_x$.
Using  Lemma \ref{proprieta slice} again, we then find a point $w \in Z^H\cap \pi_Z^{-1}(z)$ such that $\F_w \cong F'$. Since $x, y$ and $w$ are all mapped to the same point under $\eta'$, it is not restrictive to assume that $y=w$. Set $Q=\Quot^{ss}_{H_0}$ and let $\pi_Q:Q \to Q  \sslash \GL(\VV)$ be the quotient morphism. Let $Q^H \subset Q$ be the open subset parametrizing $H$--semistable sheaves. By construction, the orbit $\GL(\VV) \cdot y$ is contained in the closure of $\GL(\VV) \cdot x$. Moreover, by considering a Jordan--H\"older filtration of $\F_x$ with respect to $H$, we can proceed as in Lemma 4.4.3 of \cite{Huybrechts-Lehn} and find a one--parameter subgroup of $\GL(\VV)$ that converges to a point in the orbit of $y$. To achieve this, we only have to notice that the sheaves of the Jordan--H\"older filtration of $\F_x$ are $H$-semistable, hence $H_0$--semistable of same reduced Hilbert polynomial as $\F_x$. In particular, we can assume that they are globally generated. This also shows that the orbit $\GL(\VV) \cdot y$ is contained in $\ov{\GL(\VV) \cdot x} \cap Q^H$.

Now look at (\ref{drezet 2}), with $\Gamma=\GL(\VV)$,  $Y=Q$, and $Z$ equal to the slice at the point $q_0$. The morphism $\sigma$ restricts to a dominant morphism
\[
\GL(\VV) \times \ov {G \cdot x} \to \ov{ \GL(\VV) \cdot x} \subset \pi_Q^{-1}(\epsilon(z)).
\]
This morphism is  surjective since, in fact, $\GL(\VV) \times \ov {G \cdot x} \subset \GL(\VV) \times Z$ is a closed $G$--invariant subset so its image under the quotient morphism $\sigma$ is closed. Since $\sigma$ separates $G$--invariant closed subsets,  $\ov{G \cdot  y}$ and $\ov{G \cdot x}$
intersect in $Z$ if and only if $\ov{\GL(\VV) \cdot y}$ and $\ov{\GL(\VV) \cdot  x}$ intersect in $Q$. On the other hand, if $\ov{\GL(\VV) \cdot y}$ and $\ov{\GL(\VV) \cdot x}$ intersect in $Q^H$, then $\ov{G \cdot  y}$ and $\ov{G \cdot x}$ have to intersect in $Z^H$ and hence the lemma is proved.
\end{proof}

\begin{proof}[Proof of Proposition \ref{MH e ZH}]
Since $Z \sslash G \to M_{H_0}(v)$ is \'etale, so is the induced morphism 
\[
M_H(v) \times_{M_{H_0}(v)} Z \sslash G \to M_H(v).
\]
 Since $M_H(v)$ is smooth,  $M_H(v) \times_{M_{H_0}(v)} Z \sslash G$ is also smooth. It is therefore enough to check that the natural morphism
\[
Z^H \sslash G \to M_H(v) \times_{M_{H_0}(v)} Z \sslash G
\]
is finite and birational. By Lemma \ref{bijective}, this morphism is bijective, and it is an isomorphism on the locus parametrizing  $H_0$--stable sheaves.
\end{proof}

By Proposition \ref{Q-is-a-cone} and (\ref{mu e k2}), there is a $G$-equivariant local analytic isomorphism
\[
\varphi:(Z,q_0) \cong (\mu^{-1}(0), 0),
\]
which yields  $G$--invariant open analytic neighborhoods
\[
\mc U \subset Z, \quad \text{ and } \quad \mc V \subset  \mu^{-1}(0),
\]
of the points $q_0$ and $0$ respectively such that
\[
\varphi: \mc U  \cong \mc V,
\]
$G$--equivariantly.

\begin{prop} \label{saturation}
Up to restricting  $\U$ and $\V$, if necessary, we can assume that the following properties hold:
\begin{enumerate}
\item The two open subsets $\U$ and $\V$ are saturated neighborhoods of $q_0$ in $Z$, and of $0$ in $\mu^{-1}(0)$, respectively. 

\item Set $\U^H=\U \cap Z^H$. The natural morphisms of analytic spaces $\U \sslash G \to Z \sslash G$ and $\U^H \sslash G \to Z^H \sslash G$ are open immersions, and together with the morphisms $Z^H \sslash G \to Z \sslash G$ and $\U^H \sslash G \to U \sslash G$, they form a cartesian diagram.

\item The space $\U \sslash G$ maps isomorphically onto its image under the \'etale map $Z \sslash G \to M_{H_0}(v)$ (and the same holds for $\U^H \sslash G$ under  $Z^H \sslash G \to M_{H}(v)$).

\end{enumerate}
\end{prop}
\begin{proof}
We start with the first property. By definition, to say that $\U$ is saturated is equivalent to saying that $\pi_Z^{-1}\pi_Z(\U)=\U$.
Since $q_0 \in \U$ has closed orbit, the open subset $\U$ intersects, and hence contains, all the $G$--orbits of $\pi_Z^{-1}(\pi_Z(q_0))$ (which is the union of all orbits that contain $G \cdot q_0$ in their closure). 
The same argument applies to any point in $\U$ corresponding to polystable sheaf (since their orbits are closed), so we only have to worry about the polystable sheaves  \emph{not} contained in $\U$. 
Let $Z_\tau$ be the stratum of $Z$ parametrizing sheaves of  a given type $\tau$ (cf. (\ref{type of poly sheaf})), and let $P_\tau \subset Z_\tau$ be the locally closed $G$--invariant subset parametrizing polystable sheaves of type $\tau$. Finally, let $P_\tau^c $ be the intersection of $P_\tau$ with the complement of $\U$. Its closure (in the usual topology) $\ov P_\tau^c$ is a $G$--invariant closed subset and therefore
\[
\pi_Z(q_0) \cap \pi_Z(\ov P_\tau^c) =\emptyset.
\]
We can therefore safely remove the closed subset $\pi_Z^{-1} \pi_Z(\ov P_\tau^c)$ from $\U$ without interfering with $\pi_Z^{-1}(\pi_Z(q_0))$. Since the set of possuble strata of polystable sheaves of $Z$ is finite, we can preform this operation until we get rid of all the points of $\U$ parametrizing sheaves whose  $S_{H_0}$--equivalence class is not entirely contained in $\U$. Then we restrict $\V$ correspondingly. Since, a priori, $\V$ could be unsaturated, we can do the same trick for $\V$, and we conclude noticing that this operation does not affect the saturation of $\U$.
As for the second part, we only have to notice that since $\U$ is saturated in $Z$ (and $\U^H$ is saturated in $Z^H$) the analytic space $\U \sslash G$ is an open subset of $Z \sslash G$ (and analogously for the restriction to the locus of $H$--stable sheaves). The statement about the cartesian diagram can be proved exactly as in Proposition \ref{MH e ZH}.
The third statement is immediate.
\end{proof}

Let $\U$ and $\V$ be as in Proposition \ref{saturation} and set
\[
\V^H=\varphi (\U) \subset \mu^{-1}(0).
\]

Consider the following commutative diagram
\be\label{big-com-diag}
\xymatrix{
Z &  \ar@{_{(}->}[l] \U \ar[r]^\sim & \V \ar@{^{(}->}[r] & \mu^{-1}(0) \\
Z^H \ar@{^{(}->}[u]  \ar[d] & \U^H \ar[d] \ar@{^{(}->}[u] \ar@{_{(}->}[l] \ar[r]^\sim & \V^H \ar[d] \ar@{^{(}->}[u] &  \\ 
Z^H \sslash G \ar[d] & \ar[d] \ar@{_{(}->}[l] \U^H \sslash G \ar[r]^\sim &  \ar[d]\V^H \sslash G  \ar@{^{(}->}[r] & \ar[d] \textbf{X}  \\
Z \sslash G &  \ar@{_{(}->}[l] \U \sslash G \ar[r]^\sim & \V \sslash G \ar@{^{(}->}[r]  & \mu^{-1}(0)  \sslash G
}
\ee
In order to prove Part ii) of the Theorem, we need to understand what to place in lieu of the ``$\textbf{X}$''.

In Section \ref{character} we will find a character $\chi$, depending on $H$, such that we can set $\textbf{X}=\mu^{-1}(0) \sslash_\chi G$. In the next two sections, we will develop some necessary tools for this aim.

These sections will develop in the following setting.

Let $H_0$ be a polarization on $S$ and  consider an $H_0$--polystable sheaf  $F=\oplus_{i=1}^sF_i\otimes V_i$, where the $F_i$'s are mutually distinct $H_0$--stable sheaves. Let
\[
v=(0, D, \chi)
\]
be its Mukai vector and consider
\[
Z \subset \Quot_{H_0}
\]
an \'etale slice passing through a point $q_0$ corresponding to $F$, as in (\ref{Q slice2}).  A point $q \in  \Quot_{H_0}$ will correspond to a surjection,
\[
q: \mc O_S\otimes {H^0(F(m))} \to F(m H_0),
\]
for some chosen large $m$.  Which $m$ to choose and the fact that we can make such a choice will be discussed in the next section.

\subsection{Remarks on stability criteria} \label{remarks stability criteria}

Following Section 4.4 of \cite{Huybrechts-Lehn} we will need the following result by Le Potier, which we state in the setting of pure dimension one sheaves.

\begin{theorem}[Le Potier, Theorem 4.4.1 of \cite{Huybrechts-Lehn}]  \label{lepotier} Set $v=(0, D, \chi) $. There exists a positive integer  $m_0$ such that 
for every $m \ge m_0$ the following are equivalent
\begin{enumerate}
\item $\mc G$ is an $H_0$--semistable sheaf with Mukai vector $v$;
\item For $m \ge m_0$, $\chi(\mc G(m H_0)) \le h^0(\mc G(m H_0)$, and for any sub-sheaf $\mc G' \subset \mc G$, setting $D'=c_1(\mc G')$, we have
\be \label{sono stanca}
\ff{h^0(\mc G'(m)}{D'\cdot H_0} \le \ff{h^0(\mc G(m)}{D\cdot H_0} .
\ee
\end{enumerate}
Moreover, equality in (\ref{sono stanca}) holds if and only if $\mc G'$ makes $\mc G$ strictly $H_0$--semistable.
\end{theorem}

In order to use the Theorem above, we need to make sure that we can twist our sheaves by a large multiple of $H_0$, without affecting the problem we are set to study. Let us be more precise.

First of all, recall that $H_0$--semistability is preserved under tensoring by $H_0$. It follows that for any $m \in \ZZ$ we have a natural isomorphism
\[
M_{H_0}(v) \to M_{H_0}(v_m), \quad \text{ where } v_m:=(0, D, \chi+m (D\cdot H_0)).
\]
From the point of view of studying the singularity of $M_{H_0}(v)$, locally around a polystable sheaf $F$, we can consider without loss of generality the moduli space $M_{H_0}(v_m)$, locally around $F \otimes \mc O_S(m H_0)$.
Moreover, one can easily check using the equation of the walls given in Proposition \ref{walls and strata} that there is a bijection between $v$--walls passing through $H_0$ and the $v_m$--walls passing through $H_0$.

However, we also need to understand what happens to the resolution $h: M_{H}(v) \to M_{H_0}(v)$ as we tensor by $\mc O_S(m H_0)$.

\begin{lem}\label{H to H'} Let $H$ be a polarization adjacent to $H_0$, and set
\[
H':=\left\{ \begin{array}{ll} H & \text{if } \chi >0 \\
t H_0 -H & \text{if } \chi < 0
\end{array} \right.
\]
so that for $t \gg 0$, $H'$ is ample and adjacent to $H_0$. 
For $m \gg 0$ there is a commutative diagram
\[
\xymatrix{
M_{H}(v) \ar[d]_h \ar[r]^{\otimes m H_0}  & M_{H'}(v_m) \ar[d]^{h_m}\\
M_{H_0}(v) \ar[r]^{\otimes m H_0}   &  M_{H_0}(v_m)
}
\]
where the horizontal morphisms are isomorphism induced by tensoring by $\mc O_S(m H_0)$ and the vertical morphisms are the usual morphisms given by Proposition \ref{walls and strata}.
\end{lem}
\begin{proof}
We only have to check that the top arrow defines a regular morphism.
So let $\mc G$ be any $H$--semistable sheaf with $v(\mc G)=v$ and let $\mc G' \subset \mc G$ be a sub-sheaf and set $\Gamma=c_1( \mc G')$. Since $\mc G$ is also $H_0$--semistable we may conclude that $\mc G (m H_0)$ is also $H_0$--semistable. There are two case. Either $\mu_{H_0} (\mc G'(m H_0) ) < \mu_{H_0} (\mc G(m H_0))$, in which case the inequality stays true also for $H'$ since this polarization is adjacent to $H_0$, or else
\be \label{mu zero}
\mu_{H_0} (\mc G'(m H_0) ) = \mu_{H_0} (\mc G(m H_0)).
\ee
To handle this case, we first introduce some notation. 
For any $0 \neq L \in \Pic(S)$, set
\[
\delta(L):=\ff{\chi}{L \cdot D} - \ff{\chi(\mc G')}{  \Gamma \cdot L},
\]
so that $\mc G$ is $L$--semistable if and only if $\delta(L)>0$.
Notice that $\delta(L)=-\delta(-L)$.
Using (\ref{mu zero}) we can see that
\[
\ff{H_0 \cdot D }{L \cdot D} - \ff{H_0 \cdot \Gamma}{ L \cdot \Gamma}=\ff{H_0 \cdot D}{\chi} \delta(H)
\]
Hence,
\[
\mu_L(\mc G(m H_0) )-\mu_L(\mc G'(m H_0) )=\underbrace{ \ff{\chi}{L \cdot D} - \ff{\chi(\mc G')}{  \Gamma \cdot L}+m \left[\ff{H_0 \cdot D }{L \cdot D} - \ff{H_0 \cdot \Gamma}{ \Gamma \cdot L}\right]}_{(*)}=\delta(L)+m \ff{H_0 \cdot D}{\chi} \delta(L).
\]
If $\mc G$ is $H$--semistable, then $\delta(H)>0$. So if $\chi >0$ we may conclude that $\mc G(m H_0)$ is also $H$--semistable.

On the other hand, if $\chi<0$ and $m$ is large enough, then $\delta(H)+m \ff{H_0 \cdot D}{\chi} \delta(H)<0$, implying that $\mc G(m H_0)$ is not $H$--semistable. Now set
\[
\Delta(L):=(L\cdot D) (L \cdot \Gamma) (*).
\]
Notice that $\Delta( \cdot)$ is a linear function of its argument,  that $\Delta( H_0)=0$, and that $\Delta( H)<0$. It follows that for any $t$,
\[
\Delta (tH_0 -H) >0,
\]
and hence for $t \gg0$ so that $H'=tH_0 -H$ is ample and adjacent to $H_0$, the sheaf $\mc G(m H_0)$ is $H'$--semistable.
\end{proof}

Using this and Theorem \ref{lepotier} we hence get,

\begin{cor} \label{coro}
Up to twisting by a sufficiently high multiple of $H_0$ (and hence replacing $F$ and $v$ appropriately) we can assume in Theorem \ref{main} that :
\begin{itemize}
\item For any $H_0$--semistable sheaf $\mc G$ with Mukai vector $v$, and any sub-sheaf $\mc G' \subset \mc G$ with $\mu_{H_0}(\mc G')=\mu_{H_0}(\mc G)$, we have $H^i(\mc G)=H^i(\mc G')=0$, for $i>0$;
\item For any sub-sheaf $\mc G' \subset \mc G$ we have
\[
\ff{h^0(\mc G')}{D'\cdot H_0} \le \ff{h^0(\mc G)}{D\cdot H_0}.
\]
Moreover, equality holds  if and only if $\mc G'$ makes $\mc G$ strictly $H_0$--semistable.
\end{itemize}
\end{cor}

We henceforth assume that the conclusions of the corollary are satisfied, and since we are free to replace $F(m H_0)$ by $F$, we set
\[
\VV:=H^0(S,F),
\]
so that $\Quot_{H_0}$, which parametrizes quotients of type $\VV \otimes \mc O_S \to F$, is acted on by $GL(\VV)$.

\subsection{Remarks on linearizations}\label{linear}   The main result of this section is Proposition \ref{from H to L}. There we prove that there is a natural linearized line bundle on $Z$ such that the locus $Z^H\subset Z$
of points $z\in Z$ for which $\F_z$ is $H$-semistable is contained in the locus of semistable points with respect to this line bundle.
This will be the bridge between $Z$ and the quiver variety $\mu^{-1}(0)$.

We start by recalling the construction and the first properties of the determinant line bundle. For more details we refer the reader to Chapter 8 of \cite{Huybrechts-Lehn}.

Let $\mathcal E$ be a family of sheaves on $S$, parametrized by a scheme $B$, and let 
\[
p: B\times S\longrightarrow B\,,\quad\text{and}\quad q: B\times S\longrightarrow S,
\]
be the two projections. Set $E_B=\E_{| \{b\} \times S}$. The group homomorphism
\[
\begin{aligned}
\lambda_\E: \Pic(S) &\longrightarrow \Pic(B) \\
H & \longmapsto \lambda_\E(H):=\det p_* (\E \otimes q^* H)
\end{aligned}
\]
defines the determinant line bundle with respect to $H$. The construction is functorial on the base, in the sense that it commutes with base change. From our point of view, one important feature of $\lambda_\E(H)$ is that, if $B$ has an action of an algebraic group $G$, then any linearization of the the family $\mc E$ induces a $G$--linearization of $\lambda_\E(H)$. In turn,  this defines, for every $b \in B$, an action of the stabilizer $\Stab_b \subset G$ on the fiber $\lambda_\E(H)_b$.
This action holds an important place in the rest of the section. For example, in the case of $\lambda_\F(\ell H)$, for some ample $H$ and some $\ell \gg 0$ so that $H^i(E_b(\ell H))=0$, for every $b \in B$ and every $i >0$, and $\lambda_\F(\ell H)_b=\det H^0(E_b (\ell H))$, the action can be described as follows:
a $G$--linearization of $\mc E$ defines, for every $b \in B$, a morphism
\be\label{stab to aut}
\Stab_b \to \Aut(E_b)
\ee
which can be composed with the natural morphism
\[
\Aut(E_b) \longrightarrow GL(H^0(E_b(\ell H)) \stackrel{\det}{\longrightarrow} GL(\det H^0(E_b(\ell H))).
\]
The action is then simply given by the natural morphism
\[
\Stab_b \longrightarrow GL (\det H^0(E_b(\ell H)))=\Aut(\lambda_\E(\ell H)_b).
\]

Recall that we are assuming that the conclusions of Corollary \ref{coro} are satisfied.

Set
\[
G:=\Aut(\oplus( F_i \otimes V_i) )=\overset s{\underset{i=1}\prod}\GL(V_i).
\]

Having fixed the point $q_0 \in \Quot_{H_0}$ corresponding to the $H_0$--polystable sheaf
\[
F=\oplus ( F_i\otimes V_i ),
\]
there is an injective morphism  $i_0: G  \to GL(\VV)$, whose image is precisely $\Stab_{q_0}$.
The universal family over $\Quot_{H_0}$ has a natural $GL(\VV)$--linearization (cf. \S 4.3 of \cite{Huybrechts-Lehn})  such that for every $q \in Q$ the morphism (\ref{stab to aut}) is the inverse of the natural isomorphism $\Aut(F_q) \to \Stab_q $.

The restriction $\F$ of the universal family over the Quot scheme to $Z \times S$ is therefore $G$--linearized and hence, for every $H \in \Amp(S)$ and every $\ell \in \ZZ$, so is the determinant line bundle $\lambda_\F(\ell H)$.

We now proceed to consider GIT with respect to the $G$--line bundles $\lambda_\F(H)$ on $Z$. First we set the notation.

\begin{notation}If $\Gamma$ is an algebraic group acting on a scheme $X$
and $L$ is an ample $\Gamma$-linearized line bundle on $X$, we denote by $X^{ss}(L, \Gamma)$ the locus of semistable points in $X$ with respect to the 
$\Gamma$-linearized line bundle $L$ and by $X^{s}(L, \Gamma)$ the locus of stable points. When $L={\mc O}_X$ and $\chi: \Gamma \to \CC$ is a character of $\Gamma$, we denote by $X^{ss}(\chi, \Gamma)$ the locus of semistable points in $X$ with respect to the $\Gamma$-linearization of ${\mc O}_X$ induced by the character $\chi$, and by $X^{s}(L, \Gamma)$ the locus of stable points. When there is no risk of confusion, we omit the group $\Gamma $ from the notation and write $X^{ss}(L), X^{s}(L), \dots$, instead. 

If $L'$ and $L'$ are two $\Gamma$--line bundles  such that $X^{ss}(L, \Gamma)=X^{ss}(L', \Gamma)$ and $X^{s}(L, \Gamma)=X^{s}(L', \Gamma)$, we say that $L$ and $L'$ are $\Gamma$--equivalent, or that they define the same GIT with respect to $\Gamma$.
\end{notation}

Let $\lambda: \CC^\times \to \Gamma$ be a one-parameter subgroup ($1$ p.s.g.) and let $x \in X$ be a point. Suppose that the limit $\lim_{t \to 0} \lambda(t) \cdot x$ exists, and denote it by $\ov x$. Then, the image of the $1$ p.s.g. is contained in $\Stab_x$ and the composition of $\lambda$ with the linearization morphism $\Stab_x \to \Aut(L_x)=\CC^\times$ defines a morphism $\CC^\times \to \CC^\times$, $t \mapsto t^n$. The integer $n$, denoted by $\omega_x(\lambda, L)$, is the weight of $\lambda$ at the point $\ov x$.

We  recall the affine version of the Hilbert--Mumford criterion (cf. \cite{King}).

\begin{prop} \label{Hilb Mum}
Let $\Gamma$ be a reductive group acting an an affine scheme $X$, let $\Gamma' \subset \Gamma$ be the kernel of the action, and let $L$ be a $\Gamma$--line bundle on $X$. Then
$x \in X^{ss}(L)$ if and only if any $1$ p.s.g. $\lambda$ for which the limit exists satisfies $\omega_x(\lambda, L) \ge 0$. Moreover, $x \in X^{s}(L)$ if and only if it is semistable and for any $1$ p.s.g. $\lambda$ for which the limit exists and $\omega_x(\lambda, L) = 0$ we have $\lambda \subset \Gamma'$.
\end{prop}

Notice that a necessary condition for a point $x$ to be semistable is that if $\lambda \subset \Gamma'$, then $\omega_x(\lambda, L) = 0$. For example, if we are considering the semi-stability with respect to the trivial line bundle, linearized by a character $\chi: \Gamma \to \CC^\times$, then a necessary condition for the existence of semistable points is that $\chi$ is trivial on the kernel of the action.

In the case of the group $G=\Aut(\oplus (F_i \otimes V_i ))$ acting on the \'etale slice $Z$, the center $\CC^\times \subset G$ acts trivially on every point, but from Lemma \ref{character 1} will see that it acts with non-trivial weight.
This is a point  analogous  to the one observed in  Remark \ref{center acts trivially}. In the context of  Quot schemes and  moduli spaces,  the problem is solved by restricting the action to the subgroup of elements with trivial determinant
\[
G':=G \cap SL(\oplus V_i),
\]
so that there are no one parameter subgroups contained in the kernel of the action.
Clearly any $G$--linearization restricts to a $G'$--linearization, so that we can consider the determinant line bundle $\lambda_\F(\ell H)$ as a $G'$--line bundle.
We set
\[
Z^{ss}_{\ell H}:=Z^{ss}(\lambda_\F(\ell H), G'), \quad \text{and} \quad Z^{s}_{\ell H}:=Z^{s}(\lambda_\F(\ell H), G')
\]

We are ready to state the main result of the section.

\begin{prop} \label{from H to L}
Let $Z$, $\F$, $G$ and $G'$ be as above. Let $H$ be an ample line bundle in a chamber adjacent to $H_0$, and consider $\ell \gg 0$. We consider $\lambda_\F(\ell H)$ with the $G'$-linearization as above. Then any $z \in Z$ such that  $\F_z$ is $H$--stable, is also $\lambda_\F(\ell H)$--semistable, i.e., there is an inclusion
\[
Z^H \subset Z^{ss}_{\ell H}.
\]
Moreover, this inclusion is saturated.
\end{prop}

To prove Proposition \ref{from H to L}, we need a few lemmas, which are the adaptation of the treatment of \S 4.3 of \cite{Huybrechts-Lehn}, which we follow closely, to our context. These will lead to Corollary \ref{cor 1} and Corollary \ref{cor 2} which, together, prove the proposition. 

The notation will be as follows.
For a point $z \in Z$ we let
\[
\rho: \VV\otimes \mc O_S \to \F_z
\]
be the corresponding quotient.
For a subspace $\VV' \subset \VV$, we let
\[
\F_z'=\rho(\VV') \subset \F_z
\]
be the sub sheaf generated by $\VV'$ and we set
\[
D':=c_1(\F_z').
\]
Finally, for any sheaf $E$ on $S$ and any ample $H$ we let
\[
P_H(E,\ell)=\chi(E (\ell H)),
\]
be the Hilbert polynomial of $E$ with respect to $H$.

In order to use the Hilbert--Mumford criterion, we need to understand the limits of the one parameter subgroups of $G'$. Following \cite{Huybrechts-Lehn}, we set up the following notation. For any one parameter subgroup $\eta : \CC^\times \to G'$ let $\VV=\oplus_{\alpha \in \ZZ} \VV_\alpha$ be its weight decomposition. If
$z \in Z$ is a point corresponding to a surjection $\rho:  \VV \otimes \mc O_S \to \F_z$, we set $\VV_{\ge \alpha }= \oplus_{\beta \ge \alpha} \VV_\beta$ and define
\[
(\F_z)_{ \le \alpha}=\rho(\VV_{\ge \alpha } \otimes \mc O_S) \subset \F_z, \quad \text{and} \quad (\F_z)_ \alpha=(\F_z)_{ \le \alpha}/ (\F_z)_{ \le \alpha-1 }
\]

\begin{lem} \label{action of stab}
Let $z \in Z$ be a point and let $\eta: \CC^\times \to \Stab_z$ be a one--parameter subgroup. Then, 
\be \label{weight}
\omega_z(\eta, \lambda_\F(\ell H)=\sum \alpha P_H( (\F_z)_ \alpha, \ell ).
\ee
\end{lem}
\begin{proof}
This is Lemmas 4.4.3 and 4.4.4 of \cite{Huybrechts-Lehn}.
\end{proof}

\begin{lem} There exist an $\ell_0$ such that, if $z \in Z$ is a point  such that for any subspace $\VV' \subset \VV$ we have
\be \label{inequality}
(\dim \VV )(D' \cdot H)  \ge (\dim \VV')( D \cdot H),
\ee
then $z \in Z^{ss}_{\ell H}$ for all $\ell \ge \ell_0$. Moreover, if strict inequality holds in (\ref{inequality}), then $z \in Z^{s}_{\ell H}$ for all $\ell \ge \ell_0$
\end{lem}
\begin{proof}
Since $D \cdot H$ and $D' \cdot H$ are the coefficients of the leading terms of the Hilbert polynomials of $\F_z'$ and $\F_z$ respectively, and since the set of sheaves of the form $\rho(\VV') \subset \mc F_z$, for $\VV' \subset \VV$ and $z \in Z$ is bounded, there exist an $\ell_0$ such that (\ref{inequality}) is equivalent to
\[
\dim \VV P_H(\F_z',\ell)  \ge \dim \VV' P_H(\F_z,\ell), \quad \text{ for } \ell \gg 0
\]
(and analogously for strict inequality).
The Lemma then follows from Proposition \ref{Hilb Mum} as in the ``only if'' direction of Lemma 4.4.5 of \cite{Huybrechts-Lehn}, using the weight description (\ref{weight}).
\end{proof}

Notice that there could be one-parameter subgroups of $G'$ that do not admit a limit in $Z$, so that we cannot claim the validity of  the reverse  direction in  loc. cit. (i.e., that if $z \in Z$ is semistable then (\ref{inequality}) holds). Our claim is that we can do so as soon as the point $z$ lies in $Z^H$.

\begin{lem} \label{lemma V'} 
Let $H$ be an ample line bundle adjacent to $H_0$.
If $z \in Z^H$ is a point corresponding to an $H$--semistable sheaf, then for any subspace $\VV' \subset \VV$ we have
\be \label{inequality 1}
(\dim \VV) (D' \cdot H)  \ge (\dim \VV')( D \cdot H).
\ee
Moreover, if $z$ corresponds to an $H$--stable sheaf, then strict inequality holds in (\ref{inequality 1}).
\end{lem}
\begin{proof}
Let $\VV' \subset \VV$ be a subspace and let $\F_z' \subset \F_z$ be the sheaf generated by $\VV'$. Since $\F_z$ is $H_0$--semistable, we can apply Corollary \ref{coro} and conclude that
\be \label{inequality 2}
\ff{\dim \VV'}{D'\cdot H_0} =\ff{h^0(\F_z')}{D'\cdot H_0} \le \ff{h^0(\mc F_z)}{D\cdot H_0}=\ff{\dim \VV}{D\cdot H_0},
\ee
and that equality holds if and only if $\F_z'$ has the same reduced $H_0$--Hilbert polynomial as $\F_z$. We distinguish two cases, depending on wether $\F_z'$ has the same $H_0$--slope of $\F_z$ or not. In the first case, since by the corollary $H^i(\F_z')=0$, for $i>0$, we have $\dim \VV'=h^0(\F_z')=\chi(\F_z')$. But $\F_z$ is $H$--semistable so that we get
\[
\ff{\dim \VV'}{D'\cdot H}= \ff{\chi(\F_z')}{D'\cdot H} \le \ff{ \chi}{D\cdot H}=\ff{\dim \VV}{D\cdot H},
\]
with strict inequality in the case $\F_z$ is $H$--stable.
In the second case strict inequality holds in (\ref{inequality 2}). In this case, since $H$ is adjacent to $H_0$ the strict inequality continues to hold and hence the lemma is proved.
\end{proof}

\begin{cor} \label{cor 1}
There exists an $\ell_0$ such that for $\ell \gg \ell_0$ we have $Z^H \subset Z^{ss}_{\ell H}$ (and the set of $H$--stable sheaves is contained in $Z^{s}_{\ell H}$).
\end{cor}

The following corollary ends the proof of Proposition \ref{from H to L}

\begin{cor} \label{cor 2}
Let $H \in Amp(S)$ be in a chamber adjacent to $H_0$. Then $Z^H \subset Z^{ss}_{\ell H}$ is saturated.
\end{cor}
\begin{proof} Let $z \in Z$ be a point corresponding to a $H$--semistable sheaf $\mc F_z$. Since $H$ is $v$--generic
 $\mc F_z$ is $H$--stable and by the Corollary above we know that $z \in Z^{s}_{\ell H}$. We conclude noticing that any invariant open subset contained in a GIT stable locus is automatically saturated.
\end{proof}

The last corollary of Lemma \ref{lemma V'} that we want to highlight is the following obvious result.

\begin{cor} \label{cor 3}
$Z=Z^{ss}(\ell H_0)$ and for any one parameter group $\eta: \CC^\times \to G'$ and for any $z \in Z$, we have $\omega_z(\eta, \lambda_\F(\ell H_0)) = 0$.
\end{cor}

\subsection{From the determinant line bundle to the character}\label{character}

In this section we finally describe the missing object $\text{\bf X}$ in diagram (\ref{big-com-diag}).

It is easy to check that the weight of a one-parameter subgroup at a given point is a group homomorphism from the group of linearized line bundles to $\ZZ$. By Corollary \ref{cor 3}, it then follows that for any $G'$--line bundle $L$ and any one parameter subgroup $\eta: \CC^\times \to G'$,
\[
\omega_z(\eta, L \otimes \lambda_\F(\ell H_0))=\omega_z(\eta, L ), \quad \text{ for every } \,\, z \in Z.
\]
Hence we obtain
\[
Z^{ss}(L,G')=Z^{ss}(L \otimes \lambda_\F(\ell H_0),G')
\]
(and similarly for the set of stable points).

It follows that we can consider, instead of $\lambda_\F(\ell H)$, any combination with positive coefficients of $\lambda_\F(\ell H)$ and $ \lambda_\F(\ell H_0)$ without affecting the GIT with respect to $G'$.
To see which combination to take, we first need the following lemma.

\begin{lem} \label{character 1} For any $\ell \in \ZZ$, any $H \in \Amp(S)$, and any point $z \in Z$, the action of the stabilizer $\Stab_z$ on the fiber $\lambda_\F(\ell H)_z$ is given by restriction of the character of $G$ defined by
\[
\begin{aligned}
\chi_{\ell H}: G=\overset s{\underset{l=1}\prod}\GL(V_l) & \longrightarrow \CC^\times \\
(g_1,\dots, g_s) & \longmapsto \prod_{l=1}^s\det(g_l)^{(D_l\cdot H) \ell+ \chi_l}
\end{aligned}
\]
\end{lem}

\begin{proof}
By Proposition 4.2 of \cite{Knop-Kraft-Vust} it is enough to check the formula for points $z \in Z$ with closed orbit.
The singular locus of the \'etale slice $Z$ admits a stratification that corresponds to the stratification of the singular locus of $M_{H_0}(v)$ introduced in Proposition \ref{walls and strata}. Each stratum $Z_\tau$ corresponds to a decomposition
$\tau=(k_1, \beta^{(1)}; \dots; k_r, \beta^{(r)} )$, where:
$$
\mathbf n=k_1\beta^{(1)}+\cdots +k_r \beta^{(r)}\,,\qquad \beta^{(i)}\in\ZZ^s
$$
and 
$p(\mathbf n)>\sum_{i=1}^r p(\beta^{(i)})$, and
the points $z\in Z_\tau$ with closed orbit are of the form
$$
E=(E_1\otimes U_1)\oplus \cdots\oplus (E_r\otimes U_r)\,,\qquad \dim U_j=k_j\,,\quad j=1,\dots , r
$$
where $E_j$ is an $H_0$--stable sheaf with Mukai vector
$$
v( E_j) = \beta^{(j)}_1v_1+\cdots+\beta^{(j)}_sv _s,\quad j=1,\dots, r.
$$

We can find a decomposition
$$
V_l=\overset r{\underset{j=1}\oplus}(W_{j,\,l}\otimes U_j)\,,\quad \dim W_{j,\,l}=\beta_j^{(l)}
$$
so that, up to conjugation, the injection
\be\label{inj-stab}
j: \Stab_z=\overset r{\underset{j=1}\prod}\GL(U_j)=\Aut(E)\hookrightarrow \Aut(F)=\overset s{\underset{l=1}\prod}\GL(V_l)=\Stab_{z_0}
\ee
is given by
$$
(h_1,\dots, h_r)\quad\longmapsto\quad \left(\overset r{\underset{j=1}\oplus}(\mathbf 1_{\beta_j^{(1)}}\otimes h_j),\dots, \overset r{\underset{j=1}\oplus}(\mathbf 1_{\beta_j^{(s)}}\otimes h_j)\right),
$$
According to Lemma \ref{action of stab}, the stabilizer $\Stab_z$ acts on the fiber of 
 $ L_H$ at $z$, via the character
 $$
\chi_\tau((h_1,\dots, h_r))=\prod_{i=1}^r\det(h_i)^{ (\Delta_i\cdot H) \ell +\chi(E_i) }
$$
(here we are using notation (\ref{decompD})).
Thus we must simply show that, under the injection $j$, the character
$\chi_H$ restricts to $\chi_\tau$. This is obvious:
$$
\begin{aligned}
j^*(\chi_H)((h_1,\dots, h_r))&=\prod_{l=1}^s\left(\prod_{j=1}^r(\det h_j)^{\beta_j^{(l)}}\right)^{(D_j\cdot H) \ell +\chi_j}= \\
&=\prod_{j=1}^r\left(\prod_{l=1}^s(\det h_j)^{\beta_j^{(l)}[(D_j\cdot H) \ell+ \chi_j]}\right)=
\prod_{j=1}^r\det(h_j)^{\ (\Delta_j\cdot H) \ell +\chi(E_j)}\,.
\end{aligned}
$$
\end{proof}

Set
\[
\L_{\ell H}:= \lambda_\F(\ell H)^{D\cdot H_0}\otimes \lambda_\F(\ell H_0)^{-D\cdot H},
\]
and
\[
d=D\cdot H, \quad \text{ and } d_0=D\cdot H_0.
\]

We have already noticed that GIT with respect to the $G'$--line bundle $\lambda_\F(\ell H)$ is equivalent to GIT with respect to $\L_{\ell H}$

Under the obvious identification of $\Hom(G, \CC) \cong \ZZ^s$, the lemma above shows that the action of the stabilizer $\Stab_z$ on the fiber $(\L_{\ell H})_z$ is given by the character of $G$
\be \label{il carattere}
( \ell[d_0 (D_1\cdot H)- d (D_1\cdot H_0)], \dots,   \ell[d_0(D_s\cdot H)- d(D_l\cdot H_0)] ).
\ee

\begin{lem} \label{G e G'}
The weight of the center $\CC^\times \subset G$ with respect to $\L_{\ell H}$ is trivial and hence
\[
Z^{ss}({\L_{\ell H}, G'})=Z^{ss}(\L_{\ell H}, G),
\]
and similarly for the stable loci.
\end{lem}

It follows that we can consider $\L_{\ell H}$ as a $G$--line bundle, still without changing the GIT on $Z$.

Set
\[
\chi_H:=( [d_0 (D_1\cdot H)- d (D_1\cdot H_0)], \dots,   [d_0(D_s\cdot H)- d(D_l\cdot H_0)] ).
\]
Notice that we put the coefficients $d_0$ and $d$ so that
\[
\chi_H \in \mathbf n^ \perp.
\]

The final step shows that, in diagram (\ref{big-com-diag}), we can put $\text{\bf X}= \mu^{-1}(0)^{ss}(\chi_H, G)$.

\begin{lem}
$Z^{ss}(\L_{\ell H}, G)=Z^{ss}(\chi_{H}, G)$ and
\[
\V^H \subset \mu^{-1}(0)^{ss}(\chi_H, G).
\]
\end{lem}
\begin{proof}
The first statement is a consequence of Lemma \ref{character 1} and Lemma \ref{G e G'} above. As for the second statement, it uses the fact that both $\U \subset Z$ and $\V \subset  \mu^{-1}(0)$ are saturated open subsets. Indeed, this guarantees that if the limit of a point of $\U$ under a one parameter subgroup converges in $Z$, then it converges in $\U$, and same for the points of $\V$. It follows that the points of $\V^H$ satisfy the Hilbert--Mumford criterion for the $G$--linearization of trivial bundle given by the character $\chi_H$, and hence the lemma follows.
\end{proof}

\begin{rem}\emph{
By Remark \ref{dual} we known that, from the point of view of the resolution of quiver variety $\mathfrak M_0$, considering a character or its inverse does not matter. Hence, we are free to ignore the change of polarization given in Lemma \ref{H to H'}, which under the morphism of Lemma \ref{affine morphism} simply corresponds to taking the inverse character.
}
\end{rem}

\begin{proof}[Proof of Theorem \ref{main}]
Since there is a commutative diagram 
\[
\xymatrix{
\ar[d]  \U^H \sslash G \ar[r]^\sim &  \ar[d]\V^H \sslash G \\
\U \sslash G \ar[r]^\sim & \V \sslash G 
}
\]
We only need to check that
\[
\xymatrix{
\ar[d]\V^H \sslash G \ar@{^{(}->}[r] & \mu^{-1}(0)\sslash_{\chi_H} G \ar[d] \\
\V \sslash G  \ar@{^{(}->}[r] & \mu^{-1}(0)\sslash G
}
\]
is cartesian, but this follows exactly as in Proposition \ref{saturation} or Proposition \ref{MH e ZH}.
\end{proof}

\bibliographystyle{abbrv}

\bibliography{bibliografia}

\end{document}